\documentclass[12pt]{article}

\usepackage[T1]{fontenc}
\usepackage[utf8]{inputenc}
\usepackage[margin=1.1in]{geometry}
\usepackage{amsmath,amssymb,mathtools}
\usepackage{amsthm}
\usepackage[english]{babel}
\usepackage{graphicx}
\usepackage{xcolor}
\usepackage{float}
\usepackage[section]{placeins}
\usepackage[numbers,sort&compress]{natbib}
\usepackage[hidelinks]{hyperref}
\usepackage{bookmark}
\usepackage{adjustbox}
\usepackage{tikz-cd}
\usepackage{circuitikz}
\usepackage{comment}
\usepackage{authblk}

\usepackage{chngcntr}
\newcommand{\GLESfigure}[1]{%
  \begin{adjustbox}{width=1.30\linewidth,max totalheight=0.9\textheight,center}%
    #1%
  \end{adjustbox}%
}

\newtheorem{theorem}{Theorem}[section]
\newtheorem{lemma}[theorem]{Lemma}
\newtheorem{prop}[theorem]{Proposition}

\newtheorem{cor}[theorem]{Corollary}
\theoremstyle{definition}

\theoremstyle{remark}
\newtheorem{rem}[theorem]{Remark}
\theoremstyle{plain}

\newcommand{\Q}{\mathbb Q}

\newcommand{\Z}{\mathbb Z}
\newcommand{\F}{\mathbb F}
\newcommand{\res}{\mathrm{res}}

\makeatletter
\renewcommand{\verbatim@font}{\normalfont\footnotesize\ttfamily}
\makeatother

\title{An analog of Quillen--Venkov for eventually negligible classes and radicality for finite irreducible Coxeter groups}

\author{Constantin Cedillo Vayson de Pradenne}
\author{Matthew Gherman}
\author{Enoch Guo}
\author{Colin La}
\affil{California Institute of Technology, Department of Mathematics, Pasadena, CA 91125, U.S.A.}

\date{}

\begin{document}

\maketitle

\begin{abstract}
Let $G$ be a finite group, $M$ a $G$-module, and $F$ a field.
A continuous group homomorphism $i : \Gamma_L=\text{Gal}(L_{\text{sep}}/L) \rightarrow G$ yields a homomorphism $i^* : H^d(G,M) \rightarrow H^d(L,M)$ for every $d\in \mathbb{N}$.
An element $u \in H^d(G,M)$ is \textit{negligible over $F$} if $u\in\ker(i^*)$ for every $i$ and every field extension $L/F$.
If $M=\mathbb{F}_p$ for a prime $p$ with a trivial $G$-action, an element $u$ is \textit{eventually negligible over $F$} if $u^k$ is negligible over $F$ for some $k\geq1$.
The (eventually) negligible elements over $F$ form an ideal of the cohomology ring $H^*(G,\mathbb{F}_p)$, called the \textit{(eventually) negligible ideal}.
We prove an analog of Quillen--Venkov for eventually negligible elements.
An element $u\in H^*(G,\mathbb{F}_p)$ is eventually negligible over $F$ if and only if the restriction of $u$ to the cohomology of each elementary abelian $p$-subgroup of $G$ is eventually negligible over $F$.
Paired with applications of a Gysin-type long exact sequence, we identify generators of the negligible ideal over $\Q$ of the mod-2 cohomology of dihedral groups, cyclic $2$-groups, and generalized quaternion groups.
We prove that the negligible ideal over $\Q$ of the mod-2 cohomology of every finite irreducible Coxeter group is radical.
The result was previously known for Weyl groups; we prove it for the remaining non-Weyl types $I_2(m)$, $H_3$, and $H_4$.
\end{abstract}

\newpage
\tableofcontents

\section{Introduction}

\subsection{Background}

The notion of negligible cohomology was introduced by J.-P. Serre in \cite{Se} (see also \cite{SerreBook}).
Let $G$ be a finite group, $M$ a $G$-module, and $F$ a field.
A continuous group homomorphism $j:\Gamma_L=\text{Gal}(L_{\text{sep}}/L)\to G$ from the absolute Galois group $\Gamma_L$ of a field extension $L$ of $F$ to $G$ yields a homomorphism $j^*:H^d(G,M)\to H^d(L,M)$ of cohomology groups for every $d\geq0$.
An element $u\in H^d(G,M)$ is \textit{negligible} over $F$ if $u\in\ker(j^*)$ for all field extensions $L/F$ and all $j$. If $M$ is chosen so that one can form a cohomology ring via the cup product, an element is called \textit{eventually negligible} if some power of it is negligible. It is clear that the negligible elements over $F$ form a subgroup $H^d(G,M)_{\text{neg}}\subset H^d(G,M)$.

Negligible cohomology elements are related to the embedding problem.
Let $K/F$ be a finite Galois field extension with $G=\text{Gal}(K/F)$.
Let
\begin{equation}\label{ext}
\begin{tikzcd}
    1\arrow{r} &M\arrow{r} &G'\arrow{r}{f} &G\arrow{r} &1
\end{tikzcd}
\end{equation}
be an exact sequence of finite groups with $M$ abelian.
The conjugation $G'$-action on $M$ makes $M$ a $G$-module.
The embedding problem for the exact sequence \eqref{ext} and field extension $K/F$ is to find a Galois $G'$-algebra $K'$ over $F$ such that the restriction map $G'=\text{Gal}(K'/F)\to\text{Gal}(K/F)=G$ coincides with $f$.
Equivalently, one needs to find a lifting $\Gamma_F\to G'$ of the homomorphism $\Gamma_F\to G$ corresponding to the extension $K/F$.
Let $u\in H^2(G,M)$ be the class of the exact sequence \eqref{ext} and let $j:\Gamma_L\to G$ be the group homomorphism given by a field extension $L/F$.
Then $j$ extends to a homomorphism $\Gamma_L\to G'$ if and only if the pull-back of the sequence~\eqref{ext} under $j$ is split.
The latter is equivalent to the triviality of the image of $u$ under $j^*:H^2(G,M)\to H^2(L,M)$.
In other words, the class $u$ is negligible if and only if all embedding problems for the exact sequence \eqref{ext} and all $G$-Galois field extensions $L'/L$ of fields containing $F$ have solutions.
Recent progress has been made in \cite{MS24}. 
In Theorem~1.3, the authors determined a technique to compute degree two negligible classes with a non-trivial action on the coefficient module when $F$ contains sufficient roots of unity. 
The authors applied the technique to disprove a conjecture of Khare \cite{KL20}.

Negligible cohomology elements of $G$ are related to the invariants of $G$ as follows.
Let $M$ be an abelian group with trivial $G$-action.
Write $\text{Inv}^d(G,M)$ for the group of degree $d$ (normalized) invariants of $G$ with values in $M$ over a field $F$ (for the definition of the invariant see \cite{SerreBook}).
We have a homomorphism
\begin{equation*}
    \text{inv}:H^d(G,M)\to\text{Inv}^d(G,M),
\end{equation*}
taking an element $u\in H^d(G,M)$ to the invariant sending the class of a $G$-algebra $N$ over a field extension $L$ of $F$ (that is a $G$-torsor over Spec $L$) to the image of $u$ under the homomorphism
\begin{equation*}
    j^*:H^d(G,M)\to H^d(L,M)
\end{equation*}
with respect to the natural group homomorphism $j:\Gamma_L\to G$.
By the very definition of negligible elements, $H^d(G,M)_{\text{neg}}=\ker(\text{inv})$.

Let $M$ be an arbitrary $G$-module.
The group $H^d(G,M)_{\text{neg},F}$ is trivial if $d=0$ or $d=1$ by \cite{GM}.
Under mild hypotheses for $F$, the degree two negligible cohomology was determined in \cite{MS24}.
If $M$ is a trivial $G$-module, the degree two negligible cohomology was determined in \cite{GM}.
Serre provided explicit descriptions of mod-2 negligible classes over $\Q$ of elementary abelian 2-groups in \cite[Lemma 26.4]{SerreBook}.
The results were later extended in \cite[Theorems 2.1 and 2.3]{GhermanMerkurjevJPAA} where the mod-$p$ negligible over $F$ cohomology of elementary abelian $p$-groups was computed over an arbitrary field $F$.
In fact, \cite{GhermanMerkurjevJPAA} proves that the negligible ideal in $H^*(G,\mathbb{F}_p)$ is quite large.
In particular, the quotient ring by the ideal is finite if $F$ is not formally real or $G$ has odd order and otherwise it has Krull dimension 1.

Serre's original definition of negligible class in \cite{SerreBook} is equivalent to our definition of negligible over $\Q$.
By \cite[\S26]{SerreBook} and \cite[Proposition~2.3(2)]{GM}, a class that is negligible over $\Q$ is negligible over any other field $F$.
In what follows, unless noted otherwise, we take $M=\mathbb F_2$ with trivial action by $G$ and the ground field $F=\mathbb Q$.  

\subsection{Notation and definitions}

For $d\geq 0$ write the graded cohomology ring (with the cup product, graded-commutative) as $H^*(G, \mathbb F_2):=\bigoplus_{d\geq0}H^d(G, \mathbb F_2)$. We formally define the negligible classes $u$ and ideals $I_F$ by
\begin{align*}
    I_F^d(G,\mathbb F_2)&:=\{\,u\in H^d(G, \mathbb F_2):\ j^*(u)=0\ \text{for all }L/F\text{ and all }j:\Gamma_L\to G\,\} \\
    I_F(G,\mathbb F_2)&:=\bigoplus_{d\geq0} I_F^d(G,\mathbb F_2)\subseteq H^*(G, \mathbb F_2),
\end{align*}
and the eventually negligible ideal by
\begin{align*}
    N_F(G,\mathbb F_2)&:=\{\,x\in H^*(G, 
    \mathbb F_2):\ \exists\,k\geq1\ \text{with}\ x^k\in I_F(G,\mathbb F_2)\,\}\subseteq H^*(G, \mathbb F_2).
\end{align*}
Both $I_F(G,\mathbb F_2)$ and $N_F(G,\mathbb F_2)$ are ideals in the graded-commutative ring $H^*(G, \mathbb F_2)$. Writing $\sqrt{J}$ for the radical of an ideal $J\subseteq H^*(G, \mathbb F_2)$, we have
$$
N_F(G,\mathbb F_2)=\sqrt{\,I_F(G,\mathbb F_2)\,},
$$
or, equivalently, an element is eventually negligible if and only if some power lands in $I_F(G,\mathbb F_2)$.
We say the negligible ideal is \textit{radical} if $I_F(G,\mathbb F_2)=N_F(G,\mathbb F_2)$.
We have $F=\mathbb Q$ throughout, so we abbreviate
$$
I(G,\mathbb F_2):=I_{\mathbb Q}(G,\mathbb F_2), \qquad N(G,\mathbb F_2):=N_{\mathbb Q}(G,\mathbb F_2).$$
For a subgroup $E\le G$ we use the standard restriction and corestriction maps
$$
\mathrm{res}_E^G:H^*(G,\mathbb F_2)\to H^*(E,\mathbb F_2),\qquad \mathrm{cor}_E^G:H^*(E,\mathbb F_2)\to H^*(G,\mathbb F_2),
$$
abbreviated to $\mathrm{res}$ and $\mathrm{cor}$ when the groups are clear.

It is known that the (eventually) negligible classes of the cohomology ring $H^*(G,M)$ in fact form an ideal \cite{GhermanMerkurjevJPAA}. We now give special attention to the negligible cohomology of Coxeter groups. A \emph{Coxeter group} is a group $ W $ defined by a presentation of the form
\[
W = \left\langle s_1, s_2, \dots, s_n \,\middle|\, (s_i s_j)^{m_{ij}} = 1 \text{ for } 1 \le i,j \le n \right\rangle,
\]
where $ m_{ii} = 1 $ for all $ i $, so $ s_i^2 = 1 $, and $ m_{ij} = m_{ji} \in \mathbb{Z}_{\ge 2} \cup \{\infty\} $ for $ i \ne j $. Note that $(s_is_j)^{m_{ij}}=1$ whenever $m_{i,j}<\infty$. 
One goal of this paper is to prove the following.
\begin{theorem}\label{conjecture}
       All finite irreducible Coxeter groups have radical mod-$2$ negligible ideal over $\Q$.
\end{theorem}
In the setting of negligible cohomology, this is precisely saying that the eventually negligible classes are actually negligible. The result was previously known for Weyl groups.

\begin{theorem}[cf. Theorem 7 of \cite{Ducoat}]\label{Weylneg}
    All Weyl groups have radical mod-$2$ negligible ideal over $\Q$.
\end{theorem}

Note that the class of Coxeter groups subsumes the class of Weyl groups. By the classification of finite irreducible Coxeter groups, the only irreducible Coxeter groups that are not Weyl groups are the family $I_2(m)\cong D_{2m}$ (the dihedral groups of order $2m$ for $m\neq 2,3,4,6$) and the exceptional cases $H_3$ and $H_4$.
As an application of our main theorem, we prove the result for all the remaining non-Weyl classes.

\subsection{Outline}

 In Section~\ref{Quillen-Venkov section}, we prove our main result, an analog of Quillen-Venkov for eventually negligible classes.
We show that a class of $H^{*}(G,\mathbb F_2)$ is eventually negligible if and only if its restriction to the cohomology of each elementary abelian 2-subgroup is negligible.
We also discuss the Gysin long exact sequence technique used later.
Section~\ref{computation-sec3} applies this machinery.
Section~\ref{dihedral} computes the eventually negligible ideal for dihedral groups and proves that the negligible ideal is radical.
Section~\ref{computation-cyclic-2-groups} deduces the cyclic $2$-group case as a corollary.
Section~\ref{compputation-finite-abelian-2} determines the eventually negligible ideal for finite abelian $2$-groups.
Section~\ref{computation-generalized-quaternions} computes the negligible ideal for generalized quaternion groups and proves radicality.
Section~\ref{conjecture-sec-coxeter} completes the proof that every finite irreducible Coxeter group has radical mod-$2$ negligible ideal. The family $I_2(m)$ is covered by the dihedral computations, while the exceptional groups $H_3$ and $H_4$ are treated by restriction to odd-index subgroups with radical negligible ideal.
Appendix~\ref{appendix} collects the Gysin diagrams

\section{Analog of Quillen--Venkov (main theorem)}\label{Quillen-Venkov section}

First, we recall some results about elementary abelian 2-groups. The first proposition gives their cohomology ring as a polynomial ring.

\begin{prop}[cf. Ch.~III, \S10; Ex.~V.2.2 of \cite{Brown}]
Let $E$ be an elementary abelian 2-group of rank $r$.
Then
\begin{equation*}
    H^*(E,\mathbb F_2)\simeq\mathbb F_2[x_1,\dots,x_r]
\end{equation*}
for $\deg(x_i)=1$ for $1\leq i\leq r$.
\end{prop}
The next two results exhaustively classify the negligible elements for an elementary abelian 2-group.
\begin{theorem}[cf. Lemma 26.4 of \cite{SerreBook}]
\label{negideal}
    Let $E$ be an elementary abelian 2-group of rank $n$.
    The negligible ideal $I(E,\mathbb F_2)$ over $\Q$ is generated by elements of the form
    \begin{equation*}
        \{x_ix_j^2+x_i^2x_j:1\leq i\leq j\leq n\}.
    \end{equation*}
\end{theorem}

\begin{prop}[cf. Lemma 26.4 of \cite{SerreBook}]\label{elem ab 2-group vanishing}
Let $E$ be an elementary abelian 2-group of rank $r$.
An element $p(x_1,\dots,x_r)\in H^d(E,\mathbb F_2)$ is negligible if and only if it vanishes whenever the coordinates are in $\mathbb{F}_2$.
\end{prop}

We now move to the statement of our main result. In \cite{quillenvenkov}, Quillen and Venkov proved that nilpotent elements in mod-$p$ cohomology rings of $p$-groups can be detected on the cohomology of elementary abelian $p$-subgroups.

\begin{theorem}[Quillen-Venkov \cite{quillenvenkov}]\label{thm:quillen-venkov-nilpotence}
If $u\in H^*(G,\mathbb{F}_p)$ restricts to zero on every elementary abelian $p$-subgroup of $G$, then $u$ is nilpotent.
\end{theorem}

By an analogous argument in Galois cohomology, we prove the following generalization of Quillen-Venkov at the end of the section.

\begin{theorem}[Quillen-Venkov for eventually negligible classes]\label{detect on elem ab}
Let $G$ be a finite group and $F$ a field.
The class $u\in H^*(G,\mathbb{F}_p)$ is eventually negligible over $F$ if and only if the restriction of $u$ to each elementary abelian $p$-subgroup is eventually negligible.
\end{theorem}

Note that the interpretation of $I(E,\mathbb F_2)$ in Proposition \ref{elem ab 2-group vanishing} proves that $I(E,\mathbb F_2)$ is radical by the finite field analogue of Nullstellensatz.
Using this, we can equivalently state the conclusion of Theorem \ref{detect on elem ab} as the following.

\begin{cor}[Quillen-Venkov for mod-2 eventually negligible classes]\label{quillen-venkov}
Let $G$ be a finite group.
The class $u\in H^*(G,\mathbb F_2)$ is eventually negligible over $\Q$ if and only if the restriction of $u$ to each elementary abelian 2-subgroup is negligible.
\end{cor}

In order to prove Theorem \ref{detect on elem ab}, we need the following results about restriction and corestriction.

\begin{lemma}[cf. Proposition 2.3.1 of \cite{GM}]\label{half conjecture}
Let $H\subset G$.
Then the ring map
\begin{equation*}
    \text{res}:H^*(G,\mathbb{F}_p)\to H^*(H,\mathbb{F}_p)
\end{equation*}
sends negligible elements to negligible elements.
\end{lemma}

\begin{lemma}\label{lemma:corestriction}
Let $H\subset G$ be a subgroup of a finite group $G$.
Denote by
\begin{equation*}
    \text{cor}^d:H^d(H,\mathbb{F}_p)\to H^d(G,\mathbb{F}_p)
\end{equation*}
the degree $d$ corestriction (or transfer) map.
If $u\in H^d(H,\mathbb{F}_p)$ is negligible over $F$, then $\text{cor}^d(u)$ is negligible over $F$ in $H^d(G,\mathbb{F}_p)$.
\end{lemma}

\begin{proof}
Let $K=F(V)^G$ for $V$ a finite-dimensional faithful $F$-representation of $G$.
Since $H$ is a subgroup of $G$, $V$ is likewise a faithful $F$-representation of $H$.
Define $K_H=F(V)^H$.
We build the following commutative square.
\begin{equation*}
\begin{tikzcd}
    H^d(H,\mathbb{F}_p)\arrow{r}{\text{cor}^d}\arrow{d}{\text{inf}_H} &H^d(G,\mathbb{F}_p)\arrow{d}{\text{inf}} \\
    H^d(K_H,\mathbb{F}_p)\arrow{r}{\text{cor}^d} &H^d(K,\mathbb{F}_p)
\end{tikzcd}
\end{equation*}
By \cite[Proposition 2.1]{GM}, the kernel of these inflation maps are the negligible classes of $H^d(H,\mathbb{F}_p)$ and $H^d(G,\mathbb{F}_p)$ respectively.
Then $\text{inf}_H(u)=0$ implies that $\text{inf}(\text{cor}^d(u))=0$.
We conclude $\text{cor}^d(u)$ is negligible in $H^d(G,\mathbb{F}_p)$.
\end{proof}

\begin{lemma}\label{intersection lemma}
Let $G$ be a finite group and $p$ be a prime integer.
Let $H\subset G$ be a subgroup for which $\gcd([G:H],p)=1$.
Then restriction is an injection of rings $H^*(G,\mathbb{F}_p)\to H^*(H,\mathbb{F}_p)$.
Further,
\begin{equation*}
    I(G,\mathbb{F}_p)=H^*(G,\mathbb{F}_p)\cap I(H,\mathbb{F}_p)
\end{equation*}
when $H^*(G,\mathbb{F}_p)$ is viewed as a subring of $H^*(H,\mathbb{F}_p)$.
\end{lemma}

\begin{proof}
The composition of the corestriction and restriction maps is multiplication by $[G:H]$ on each $H^d(G,\mathbb{F}_p)$.
Since $\text{gcd}([G:H],p)=1$, the composition is an isomorphism on each $H^d(G,\mathbb{F}_p)$.
We conclude that $\text{res}:H^*(G,\mathbb{F}_p)\to H^*(H,\mathbb{F}_p)$ is an injective ring map and each $\text{cor}^d:H^d(H,\mathbb{F}_p)\to H^d(G,\mathbb{F}_p)$ is a surjective group homomorphism.

View $H^*(G,\mathbb{F}_p)$ as a subring of $H^*(H,\mathbb{F}_p)$ via restriction.
If $u\in I(G,\mathbb{F}_p)$, then $\text{res}(u)\in I(H,\mathbb{F}_p)$ by \cite[Proposition 2.3(1)]{GM}.
Thus
\begin{equation*}
    I(G,\mathbb{F}_p)\subset H^*(G,\mathbb{F}_p)\cap I(H,\mathbb{F}_p).
\end{equation*}
If $v\in H^d(G,\mathbb{F}_p)\cap H^d(H,\mathbb{F}_p)_{\text{neg}}$, then $v=\text{res}(u)$ for some $u\in H^d(G,\mathbb{F}_p)$.
We have $\text{cor}^d(v)=\text{cor}^d(\text{res}^d(u))=[G:H]u$ is an element of $H^d(G,\mathbb{F}_p)_{\text{neg}}$ by Lemma \ref{lemma:corestriction}.
We conclude that $u$ is an element of $H^d(G,\mathbb{F}_p)_{\text{neg}}$ since $[G:H]$ is invertible in $H^d(G,\mathbb{F}_p)$.
For an inhomogeneous $w\in H^*(G,\mathbb{F}_p)\cap I(H,\mathbb{F}_p)$, perform the above procedure on each homogeneous piece.
Therefore,
\begin{equation*}
    I(G,\mathbb{F}_p)\supset H^*(G,\mathbb{F}_p)\cap I(H,\mathbb{F}_p). \qedhere
\end{equation*}
\end{proof}

The next lemma records a long exact sequence which is ubiquitous in our arguments.

\begin{lemma}[Gysin long exact sequence]\label{restriction les}
Let $\Gamma$ be a profinite group with character $x_i:\Gamma\to\mathbb Z/2 \mathbb Z$.
Denote the subgroup $\Gamma_i=\ker(x_i)$ of $\Gamma$.
Then the following exact sequence exists for $d\geq1$.
\[
\begin{tikzcd}[column sep=small]
    \cdots \arrow{r}{\mathrm{cor}^{d-1}}
    & H^{d-1}(\Gamma,\mathbb F_2)
      \arrow{r}{\cup x_i}
    & H^d(\Gamma,\mathbb F_2)
      \arrow{r}{\mathrm{res}}
    & H^d(\Gamma_i,\mathbb F_2)
      \arrow{r}{\mathrm{cor}^d}
    & H^d(\Gamma,\mathbb F_2)
      \arrow{r}{\cup x_i}
    & \cdots
\end{tikzcd}
\]
\end{lemma}

\begin{proof}
Endow the $\Gamma$-module exact sequence
\begin{equation*}
\begin{tikzcd}
    0\arrow{r} &\mathbb{F}_2\arrow{r} &\mathbb{F}_2\oplus\mathbb{F}_2\arrow{r} &\mathbb{F}_2\arrow{r} &0
\end{tikzcd}
\end{equation*}
with $\Gamma$-action on the middle term $\gamma\cdot(a,b)=(a+x_i(\gamma)\cdot b,b)$ for $a,b\in\mathbb F_2$ and $\gamma\in G$.
The middle term is isomorphic to $\text{Ind}_{\Gamma_i}^{\Gamma}(\mathbb Z/2\mathbb Z)$ as $\Gamma$-modules.
Shapiro's lemma applied to the induced long exact sequence completes the proof.
See \cite[Theorem 5]{Rusin} for another proof.
\end{proof}

As a first application of the above sequence, we prove a partial result classifying eventually negligible elements.
\begin{lemma}\label{index 2 subgroups}
Let $G$ be a finite 2-group that is not elementary abelian and $F$ a field.
Assume that the restriction of $u\in H^*(G,\mathbb F_2)$ is eventually negligible over $F$ in the cohomology of each index 2 subgroup of $G$.
Then $u$ is eventually negligible in $H^*(G,\mathbb F_2)$ over $F$.
\end{lemma}

\begin{proof}
Since the squaring map is linear in the ring $H^*(G,\mathbb F_2)$, it is sufficient to check the claim on homogeneous $u$.
Denote by $H_i$ the index 2 subgroups of $G$ for $1\leq i\leq n$.
Let $V$ be a faithful $F$-representation of $G$.
Define $K=F(V)^G$ and $K_i=F(V)^{H_i}$.
We obtain group homomorphisms $j:\Gamma_K\to G$ and $j_i:\Gamma_{K_i}\to H_i$.

Each index 2-subgroup $H_i$ of $G$ is the kernel of a character $x_i\in H^1(G,\mathbb F_2)$.
Then $j^*(x_i)\in H^1(K,\mathbb F_2)$ is a group homomorphism from $\Gamma_K$ to $\mathbb Z/2\mathbb Z$.
Lemma \ref{restriction les} and $j$ induce the following commutative diagram with exact rows in each degree $d\geq1$.
\begin{equation*}
\begin{tikzcd}
    H^{d-1}(G,\mathbb F_2)\arrow{r}{\cup x_i}\arrow{d}{j^*} &H^{d}(G,\mathbb F_2)\arrow{r}{\text{res}_i}\arrow{d}{j^*} &H^{d}(H_i,\mathbb F_2)\arrow{d}{j_i^*} \\
    H^{d-1}(K,\mathbb F_2)\arrow{r}{\cup j^*(x_i)} &H^{d}(K,\mathbb F_2)\arrow{r}{\text{res}_i} &H^{d}(K_i,\mathbb F_2)
\end{tikzcd}
\end{equation*}

Replace $u$ by some power of $u$ so that $u$ restricts to a negligible element over $F$ in the cohomology of each $H_i$.
Then $\text{res}_i(j^*(u))=j_i^*(\text{res}_i(u))=0$ for all $1\leq i\leq n$.
Since $\text{res}_i(j^*(u))$ is trivial for each $1\leq i\leq n$, exactness of the bottom row implies that $j^*(u)$ is divisible by $j^*(x_i)$ for each $1\leq i\leq n$.
Let $\beta:H^1(G,\mathbb F_2)\to H^2(G,\mathbb F_2)$ be the Bockstein homomorphism, which, in this case, is the squaring map.
Since $G$ is not an elementary abelian 2-group, \cite[Proposition 4]{Se} proves there is some sequence of the $x_i$ for which $\prod\beta(x_i)=\prod x_i^2=0$.
Thus, there is a positive integer $m$ for which $j^*(u)^{2m}$ is divisible by $\prod j^*(x_i)^2=\prod j^*(\beta(x_i))=j^*(\prod\beta(x_i))=0$.
Since $j^*(u)$ is nilpotent in $H^*(K,\mathbb F_2)$, $u\in H^*(G,\mathbb F_2)$ is eventually negligible over $F$. 
\end{proof}

\begin{proof}[Proof of Theorem \ref{detect on elem ab}]\label{proof-theorem-elem-ab}
Let $E$ be an elementary abelian $p$-subgroup of $G$.
In the case that $F$ is not a formally real field, \cite[Theorem 2.2]{GhermanMerkurjevJPAA} proves that each class of $H^*(G,\mathbb{F}_p)$ and each class of $H^*(E,\mathbb{F}_p)$ is eventually negligible.
For the theorem when $F$ is formally real, the case $p$ odd is also covered by \cite[Theorem 2.2]{GhermanMerkurjevJPAA} since in positive degree every class is eventually negligible.
Thus, the only remaining nontrivial case over a field $F$ that is formally real is $p=2$.
If $p=2$, $F$ is formally real, and $G$ is odd order, $G$ contains no non-trivial elementary abelian 2-subgroups and $H^{>0}(G,\mathbb F_2)$ is trivial.
The result trivially holds.
We are left to prove the case in which $p=2$, $F$ is formally real, and $G$ is even order.

By Lemma \ref{half conjecture}, we need only prove the following.
If $u\in H^*(G,\mathbb F_2)$ restricts to an eventually negligible element in $H^*(E,\mathbb F_2)$ for each elementary abelian 2-subgroup $E\subset G$, then $u$ is eventually negligible.
If the result holds for a Sylow 2-subgroup of $G$, then Lemma \ref{intersection lemma} implies that the result holds for $G$.
We assume without loss of generality that $G$ is a 2-group.

We will proceed by induction on $|G|$.
The base case is the cyclic group of order 2, for which there are no non-trivial negligible classes by \cite[Theorem 4.1]{GhermanMerkurjevJPAA}.
In the inductive step, if $G$ is an elementary abelian 2-group, the result holds by inspection \cite[Theorem 4.1]{GhermanMerkurjevJPAA}.
Lemma \ref{index 2 subgroups} proves the inductive step when $G$ is not elementary abelian.
\end{proof}

\section{Sketch of methods for detecting (eventually) negligible classes}
\subsection{Eventually negligible elements}
Theorems \ref{negideal} and \ref{detect on elem ab} will serve as our primary detection method for eventually negligible elements. Our methodology is as follows.
\begin{enumerate}
    \item We begin by taking the restriction maps $H^*(G, \mathbb F_2) \rightarrow H^*(E, \mathbb F_2)$
    for all \textit{maximal} elementary abelian subgroups $E$ of $G$.
    \item Using Theorems \ref{negideal} and \ref{detect on elem ab}, we identify candidate eventually negligible classes.
    Denote by $N$ the ideal generated by all the eventually negligible elements found in this step.
    We have that $N\subset N(G,\mathbb F_2)$ and wish to prove that $N \supset N(G,\mathbb F_2)$ to establish equality.
    \item To do this, we quotient by $N$ and consider $\mathbb F_2[x_1,\dots,x_n] / N$.
    \item Looking at elements in the quotient ring, we take the restriction map and check that there are no non-trivial elements remaining.
    These elements take the form of polynomials, so we can plug in arbitrary values from $\mathbb F_2$ to the generators $\{x_1,\dots,x_n\}$.
    By Proposition \ref{elem ab 2-group vanishing} (\cite[Lemma 26.4]{SerreBook}), an element is negligible in $H^*(E,\mathbb F_2)$ if and only if the corresponding polynomial evaluates to 0 for all $\mathbb F_2$ inputs of $\{x_1,\dots,x_n\}$.
    Therefore, we want to prove that each non-trivial element of $H^*(G,\mathbb F_2)/N$ restricts to a polynomial that does not evaluate to 0 for some choice of inputs for $\{x_1,\dots,x_n\}$ in $H^*(E,\mathbb F_2)$.
    In other words, we prove $N=N(G, \mathbb F_2)$ by showing the injectivity of the map induced by restriction on the right-hand side of the diagram below.
\end{enumerate}
\begin{equation}\label{comm diagram}
    \begin{tikzcd}
    H^*(G,\mathbb F_2)\arrow{r}\arrow{d}{\text{res}} &H^*(G,\mathbb F_2) / N\arrow[hookrightarrow]{d}{\overline{\text{res}}} \\
    H^*(E,\mathbb F_2) \arrow{r} & H^*(E,\mathbb F_2) / I(E,\mathbb F_2)
\end{tikzcd}
\end{equation}

\subsection{Negligible elements}
Once we determine the eventually negligible elements, let us now briefly sketch how the Gysin long exact sequence is used to detect actually negligible elements.
Firstly, we have a clear picture on what type of elements are not negligible.
\begin{theorem}[cf. Corollary 2.2 and Theorem 5.2 of \cite{GM}]\label{nonneg-low-degree}
    There exist no non-trivial negligible classes of degree 1.
    In addition, there exist no non-trivial negligible classes of degree 2 over $\mathbb Q$.
\end{theorem}
Thus, if any generator has degree 1 or 2, we can automatically rule it out from being actually negligible. The following result along with Lemmas \ref{half conjecture} and \ref{lemma:corestriction} are key for finding actually negligible elements.

\begin{prop}\label{corestriction identity}
    For any two classes $a,b\in H^*(G,\mathbb F_2)$ in the cohomology ring, we have
    $$\operatorname{cor}(a\cdot \operatorname{res}(b))=\operatorname{cor}(a)\cdot b$$
\end{prop}

Lastly, in degree one, the group cohomology is isomorphic to homomorphisms from the group to the coefficients. Since our group $G$ is not necessarily abelian, and we are working with abelian coefficients, we know that these Homs must factor through the commutator first, and so are in bijection with homomorphisms to $\mathbb F_2$ from the abelianization of $G$, in other words
$$H^1(G, \mathbb F_2) \cong \text{Hom}(G, \mathbb F_2) \cong \text{Hom}(G / [G, G], \mathbb F_2).$$
For all the groups in consideration with the following method, this computation is entirely tractable.

Equipped with these facts, the plan is then clear in principle.
We begin by finding an index two subgroup of $G$ whose negligible ideal is already known, and construct the long exact sequence using it.
Then we determine explicitly the restriction map in degree 1, and since restriction is a ring homomorphism, this lifts to arbitrarily high degrees in the cohomology ring, thus giving us information about higher degrees. Along with Proposition \ref{corestriction identity}, this gives us enough data to explicitly write out much of the long exact sequence.
Through a diagram chase, we can then use the criteria above to rule out elements that cannot possibly be negligible.
Finally, as we know a priori the eventually negligible elements in $G$, we check if these elements restrict or corestrict to a negligible element in $H$, thus forcing it to be negligible.
\section{Computations of the mod-2 negligible ideal of dihedral groups, cyclic 2-groups, and generalized quaternions}\label{computation-sec3}

\subsection{Negligible ideal of the mod-2 cohomology of dihedral groups}\label{dihedral}

Let $D_{2n}$ be the dihedral group of order $2n$.
When $n$ is odd, \cite[Theorem 5.6]{dihedralcohomology} proves
\begin{equation*}
    H^*(D_{2n},\mathbb F_2)\simeq\mathbb F_2[u],
\end{equation*}
where $\deg(u)=1$.

\begin{prop}\label{dihedralmodd}
Let $D_{2n}$ be the dihedral group of order $2n$ with $n$ an odd positive integer.
Then $I(D_{2n},\mathbb F_2)$ is trivial and the negligible ideal is radical.
\end{prop}

\begin{proof}
The Sylow 2-subgroup of $D_{2n}$ is isomorphic $\Z/2\Z$.
By \cite[Lemma 26.4]{SerreBook}, $I(\mathbb Z / 2\mathbb Z,\mathbb F_2)$ has trivial negligible cohomology, so $I(D_{2n},\mathbb F_2)$ is trivial by Lemma \ref{intersection lemma}.
\end{proof}

When $n\equiv2$ (mod 4), \cite[Theorem 5.5]{dihedralcohomology} proves 
\begin{equation*}
    H^*(D_{2n},\mathbb F_2)\simeq\mathbb F_2[u,v,w]/\langle u(u+v) + w\rangle,
\end{equation*}
where $\deg(u)=\deg(v)=1$ and $\deg(w)=2$.

\begin{prop} \label{dihedralm2}
Let $D_{2n}$ be the dihedral group of order $2n$ with $n\equiv2$ (mod 4).
Then $I(D_{2n},\mathbb{F}_2)=N(D_{2n},\mathbb F_2)=\langle uv(u+v)\rangle=\langle vw\rangle$ is radical.
\end{prop}

\begin{proof}
Write $n=2\ell$ with $\ell$ odd. Then $|D_{2n}|=4\ell$, and a Sylow $2$-subgroup $E\subset D_{2n}$ is elementary abelian of rank $2$ and has odd index $\ell$.
Choose degree-one generators $a,b\in H^1(E,\mathbb F_2)$ so that, after relabeling if necessary, $\res_E(u)=a$ and $\res_E(v)=b$.
By Lemma \ref{intersection lemma},
\[
I(D_{2n},\mathbb F_2)=H^*(D_{2n},\mathbb F_2)\cap I(E,\mathbb F_2).
\]
For $E\simeq(\mathbb Z/2\mathbb Z)^2$, Theorem \ref{negideal} gives
\[
I(E,\mathbb F_2)=\langle ab(a+b)\rangle.
\]
Therefore $I(D_{2n},\mathbb F_2)=\langle uv(u+v)\rangle$. Since $u(u+v)+w=0$ in this case, this ideal is also $\langle vw\rangle$.
The ideal $I(E,\mathbb F_2)$ is radical, so its preimage under restriction is radical.
\end{proof}

We now consider the case $n\equiv 0$ (mod $4$). In \cite[Theorem 5.5]{dihedralcohomology}, the mod-2 cohomology ring for $n\equiv 0$ (mod 4) is
\begin{equation*}
    H^*(D_{2n},\mathbb F_2)\simeq\mathbb F_2[u,v,w]/\langle u(u+v)\rangle,
\end{equation*}
where $\deg(u)=\deg(v)=1$ and $\deg(w)=2$.

\begin{prop}\label{dihedral even n}
    Let $D_{2n}$ be the dihedral group of order $2n$ with $n\equiv0$ (mod 4) an even positive integer.
    The eventually negligible ideal is $$N(D_{2n},\mathbb F_2) = \langle uw, vw\rangle.$$
\end{prop}
\begin{proof}
Up to conjugacy, there are three elementary abelian $2$-subgroups of $D_{2n}$, two of which are maximal \cite{unijenaCohomologyFinite}.
Let $E_1,E_2$ denote the two maximal ones of rank 2. 
With respect to generators $u,v,w$ for $H^*(D_{2n},\mathbb F_2)$ and generators $x_1,x_2$ for $H^*(E_i,\mathbb F_2)$, \cite{unijenaCohomologyFinite} computes the restriction maps:
\begin{align*}
\res_{E_1}\colon &\quad u \mapsto x_2,\qquad v \mapsto 0,\qquad w \mapsto x_1x_2+x_1^2,\\
\res_{E_2}\colon &\quad u \mapsto 0,\qquad v \mapsto x_2,\qquad w \mapsto x_1x_2+x_1^2.
\end{align*}
By Theorem~\ref{negideal}, the products $uw$ and $vw$ lie in $N(G,\mathbb F_2)$.
We claim
\[
N(G,\mathbb F_2)=\langle uw, vw\rangle.
\]
In the quotient,
\[
H^*(D_{2n},\mathbb F_2)/\langle uw,vw\rangle
\simeq \mathbb F_2[u,v,w]/\langle u(u+v),uw,vw\rangle.
\]
It remains to show that no additional non-zero class is eventually negligible.
Let $y\in H^*(D_{2n},\mathbb F_2)/\langle uw,vw\rangle$ be a homogeneous element of degree $d>0$.
Indeed, $uw=vw=0$ kills all mixed terms involving $w$, while $u(u+v)=0$ gives $u^2=uv$, so every degree-$d$ monomial in $u$ and $v$ reduces to either $u^d$ or $v^d$.
Using these relations, every such $y$ can be written as
\[
y=\alpha_1 u^d+\alpha_2 v^d+\alpha_3 w^{d/2},
\qquad \alpha_i\in \mathbb F_2.
\]
with $\alpha_3=0$ if $d$ is odd. Under restriction to $E_1$ and $E_2$ this becomes
\begin{align*}
\res_{E_1}(y) &= \alpha_1 x_2^d+\alpha_3 (x_1x_2+x_1^2)^{d/2},\\
\res_{E_2}(y) &= \alpha_2 x_2^d+\alpha_3 (x_1x_2+x_1^2)^{d/2}.
\end{align*}
Assume for contradiction that $y$ is eventually negligible.  By Theorem~\ref{detect on elem ab}, its restriction to each elementary abelian $2$-subgroup is negligible.
By Proposition~\ref{elem ab 2-group vanishing}, for an elementary abelian $2$-group this is equivalent to vanishing upon evaluation at $(x_1,x_2)\in(\mathbb Z/2\mathbb Z)^2$.
Evaluate at $(x_1,x_2)=(0,1)$. Then $x_1x_2+x_1^2=0$, so
\[
\res_{E_1}(y)(0,1)=\alpha_1,\qquad \res_{E_2}(y)(0,1)=\alpha_2.
\]
Negligibility forces both to vanish, hence $\alpha_1=\alpha_2=0$.

Assume $d$ is even.
Evaluate at $(x_1,x_2)=(1,0)$.
Then $x_1x_2+x_1^2=1$, so
\[
\res_{E_1}(y)(1,0)=\res_{E_2}(y)(1,0)=\alpha_3.
\]
Thus, $\alpha_3=0$, and $y=0$ in $H^*(D_{2n},\mathbb F_2)/\langle uw,vw\rangle$.
\end{proof}

Now, we apply the long exact sequence technique to prove the negligible ideal of a dihedral 2-group $D_{2^n}$ is radical via induction on $n$.
In Proposition \ref{dihedral even n}, we prove the eventually negligible ideal is $N(D_{2^n}, \mathbb F_2) = \langle uw, vw\rangle$. 
Thus, it suffices to show that $uw + vw$ and $uw$ are negligible in $D_{2^n}$.

\begin{rem}
To illustrate the first example of this method concretely, we display the corresponding long exact sequence in Figure \ref{GysinD2n-1}.
In the proofs throughout the rest of the paper, we refer to the diagrams in the Appendix \ref{appendix}.
\end{rem}

\begin{prop}\label{prop:uw-and-uw+vw-negligible}
Let $D_{2^n}$ be the dihedral group of order $2^n$.
For all $n\geq 3$,
\begin{equation*}
    I(D_{2^n},\mathbb{F}_2)=\langle uw, vw\rangle.
\end{equation*}
\end{prop}

\begin{proof}
We argue by induction on $n$. For the base case $n=3$, Proposition \ref{dihedral even n} gives
$N(D_8,\mathbb F_2)=\langle uw,vw\rangle.$
Since Theorem \ref{Weylneg} shows that $I(D_8,\mathbb F_2)$ is radical, it follows that $\{uw,\;vw\}\subset I(D_8,\mathbb F_2).$
Alternatively, \cite[Theorem 6.1]{BaekMerkurjev} proves the result.

Now let $n\geq 4$, and suppose inductively that
$u'w'$ and $u'w'+v'w'$
are negligible in $H^*(D_{2^{n-1}},\mathbb F_2)$.

We first prove that $uw+vw$ is negligible in $H^*(D_{2^n},\mathbb F_2)$.
Consider the index-two subgroup $D_{2^{n-1}}\subset D_{2^n}$ generated by $r^2$ and $s$, and choose
the generators $u,v\in H^1(D_{2^n},\mathbb F_2)=\mathrm{Hom}(D_{2^n},\mathbb F_2)$ by
\[
u(r)=\overline{1},\qquad u(s)=\overline{0},\qquad
v(r)=\overline{0},\qquad v(s)=\overline{1}.
\]
Let $u',v'\in H^1(D_{2^{n-1}},\mathbb F_2)$ be defined analogously on the subgroup generated by $r^2$ and $s$.
Then
\[
\mathrm{res}(u)(r^2)=u(r^2)=u(r+r)=\overline{0},\qquad
\mathrm{res}(u)(s)=u(s)=\overline{0},
\]
so $\mathrm{res}(u)=0$, while
\[
\mathrm{res}(v)(r^2)=v(r^2)=v(r+r)=\overline{0},\qquad
\mathrm{res}(v)(s)=v(s)=\overline{1},
\]
so $\mathrm{res}(v)=v'$. Thus the degree-$1$ restriction map is
\[
\mathrm{res}(u)=0,\qquad \mathrm{res}(v)=v'.
\]
Using this restriction map and the connecting homomorphism given by cup product with $u$, one obtains the long exact sequence displayed in Figure \ref{GysinD2n-1}.
Since restriction is a ring homomorphism, the degree-one computation determines the higher restriction maps appearing in Figure \ref{GysinD2n-1}.

By Figure \ref{GysinD2n-1}, the kernel of
\[
H^3(D_{2^n},\mathbb F_2)\xrightarrow{\cup u} H^4(D_{2^n},\mathbb F_2)
\]
is generated by $\{uw+vw,\;u^3+v^3\}.$
By exactness, this kernel is equal to the image of
\[
\mathrm{cor}^3:H^3(D_{2^{n-1}},\mathbb F_2)\to H^3(D_{2^n},\mathbb F_2).
\]
 Since $u'w'$ is negligible in $H^*(D_{2^{n-1}},\mathbb F_2)$, Lemma
\ref{lemma:corestriction} implies that the non-zero $\mathrm{cor}^3(u'w')$ is negligible in $H^*(D_{2^n},\mathbb F_2)$.
On the other hand, $u^3+v^3$ and $uw+vw+u^3+v^3$ are not eventually negligible, so $\mathrm{cor}^3(u'w')=uw+vw.$
It follows that $uw+vw$ is negligible in $H^*(D_{2^n},\mathbb F_2)$.

We now prove that $uw$ is negligible.
For this, we use the second description of the degree-one generators,
\[
u(r)=\overline{1},\qquad u(s)=\overline{1},\qquad
v(r)=\overline{0},\qquad v(s)=\overline{1}.
\]
Again, let $u',v'$ denote the corresponding generators of
$H^1(D_{2^{n-1}},\mathbb F_2)$.
Then
\[
\mathrm{res}(u)(r^2)=u(r^2)=u(r+r)=\overline{0},\qquad
\mathrm{res}(u)(s)=u(s)=\overline{1},
\]
so $\mathrm{res}(u)=v'$, and similarly
\[
\mathrm{res}(v)(r^2)=v(r^2)=v(r+r)=\overline{0},\qquad
\mathrm{res}(v)(s)=v(s)=\overline{1},
\]
so $\mathrm{res}(v)=v'$. Consequently,
\[
\mathrm{res}(u+v)=\mathrm{res}(u)+\mathrm{res}(v)=v'+v'=0.
\]
Thus, the degree-one restriction map for this second choice is
\[
\mathrm{res}(u)=v',\qquad \mathrm{res}(v)=v',\qquad \mathrm{res}(u+v)=0.
\]
Using this restriction map and the connecting homomorphism given by cup product with $u+v$, one obtains the long exact sequence in Figure \ref{GysinD2^n(uplusv)}.

Since the relation in the cohomology ring is $u(u+v)=0,$ we have $u^2=uv.$
By Figure \ref{GysinD2^n(uplusv)}, the kernel of
\[
H^3(D_{2^n},\mathbb F_2)\xrightarrow{\cup (u+v)} H^4(D_{2^n},\mathbb F_2)
\]
is generated by $\{uw,\;u^3\}.$
By exactness, this kernel is equal to the image of
\[
\mathrm{cor}^3:H^3(D_{2^{n-1}},\mathbb F_2)\to H^3(D_{2^n},\mathbb F_2).
\]
By the induction hypothesis, both $u'w'$ and $v'w'$ are negligible in $H^*(D_{2^{n-1}},\mathbb F_2)$, so $\mathrm{cor}^3(u'w')$ and $\mathrm{cor}^3(v'w')$ are negligible in $H^*(D_{2^n},\mathbb F_2)$ by Lemma \ref{lemma:corestriction}.
Moreover, by Proposition \ref{corestriction identity},
$$\mathrm{cor}\big((u')^3\big)=\mathrm{cor}\big(u'(v')^2\big)=\mathrm{cor}(u')u^2=u^3.$$
Hence $u^3$ is the image of $(u')^3$ under corestriction.
Since $u^3$ and $u^3+uw$ are not eventually negligible, the other generator of $\mathrm{Im}(\mathrm{cor}^3)$, namely $uw$, must be the corestriction of a negligible class.
Therefore, $uw$ is negligible in $H^*(D_{2^n},\mathbb F_2)$, which completes the induction.
\end{proof}

\begin{cor}\label{dihedral 2 is radical}
The negligible ideal of $H^*(D_{2^n},\mathbb{F}_2)$ is radical. 
\end{cor}

We extend our result to all dihedral groups, $D_{2n}$.
Propositions \ref{dihedralmodd} and \ref{dihedralm2} handle the cases when $n$ is odd and $n \equiv 2\text{ }(\text{mod}\text{ }4)$. 

\begin{prop}\label{dihedralm0}
    For $n \equiv 0$ (mod 4),
    \begin{equation*}
        I(D_{2n},\mathbb{F}_2)=\langle uw, vw\rangle.
    \end{equation*}
    The negligible ideal of $D_{2n}$ is radical.
\end{prop}
\begin{proof}
Write $2n=2^k\ell$, where $k\geq3$ and $\ell$ is odd. Then a Sylow $2$-subgroup of $D_{2n}$ is isomorphic to $D_{2^k}$. Since this subgroup has odd index $\ell$, Lemma \ref{intersection lemma} identifies the negligible ideal of $D_{2n}$ with the intersection of $H^*(D_{2n},\mathbb F_2)$ and the negligible ideal of $D_{2^k}$ under restriction. Proposition \ref{prop:uw-and-uw+vw-negligible} gives $I(D_{2^k},\mathbb F_2)=\langle uw,vw\rangle$, so $I(D_{2n},\mathbb F_2)=\langle uw,vw\rangle$. Thus the negligible ideal of $D_{2n}$ is radical.
\end{proof}

The results of Propositions \ref{dihedralmodd}, \ref{dihedralm2}, and \ref{dihedralm0} prove that the negligible cohomology ideal of $D_{2n}$ is radical.

\begin{figure}[htbp]
    \centering
\begin{center}
    \GLESfigure{%
      \begin{circuitikz}
         \tikzstyle{every node}=[font=\large]
\node [font=\Large] at (-1,16.5) {$H^1(D_{2^n}, \mathbb F_2)$};

\draw [->, >=Stealth] (1,16.5) -- (2.75,16.5);
\node [font=\Large] at (5,16.5) {$H^1(D_{2^{n-1}}, \mathbb F_2)$};
\draw [->, >=Stealth] (7,16.5) -- (8.25,16.5);
\node [font=\Large] at (10.25,16.5) {$H^1(D_{2^n}, \mathbb F_2)$};
\draw [->, >=Stealth] (12.5,16.5) .. controls (22.75,11) and (-10.25,15) .. (-3,9) ;
\node [font=\Large] at (-1,9) {$H^2(D_{2^n}, \mathbb F_2)$};
\node [font=\Large] at (5,9) {$H^2(D_{2^{n-1}}, \mathbb F_2)$};
\node [font=\Large] at (10.25,9) {$H^2(D_{2^n}, \mathbb F_2)$};
\node [font=\LARGE] at (3.25,16) {};
\draw [->, >=Stealth] (1,9) -- (2.75,9);
\draw [->, >=Stealth] (7,9) -- (8.25,9);
\draw [->, >=Stealth] (12.5,8.75) .. controls (22.75,3.25) and (-10.25,7.25) .. (-3,1.25) ;
\node [font=\Large] at (-1,1.25) {$H^3(D_{2^n}, \mathbb F_2)$};
\node [font=\Large] at (5,1.25) {$H^3(D_{2^{n-1}}, \mathbb F_2)$};
\node [font=\Large] at (10.25,1.25) {$H^3(D_{2^n}, \mathbb F_2)$};
\draw [->, >=Stealth] (1,1.25) -- (2.5,1.25);
\draw [->, >=Stealth] (7,1.25) -- (8.25,1.25);
\draw [->, >=Stealth] (12.5,1) .. controls (22.75,-5) and (-10.25,-1) .. (-3,-7.25) ;
\node [font=\Large] at (-1,-7.25) {$H^4(D_{2^n}, \mathbb F_2)$};
\node [font=\Large] at (5,-7.25) {$H^4(D_{2^{n-1}}, \mathbb F_2)$};
\node [font=\Large] at (10.25,-7.25) {$H^4(D_{2^n}, \mathbb F_2)$};
\draw [->, >=Stealth] (1,-7.25) -- (2.75,-7.25);
\draw [->, >=Stealth] (7,-7.25) -- (8.25,-7.25);
\node [font=\normalsize] at (2,17) {$\mathrm{res}^1$};
\node [font=\normalsize] at (7.5,17) {$\mathrm{cor}^1$};
\node [font=\normalsize] at (2,9.5) {$\mathrm{res}^2$};
\node [font=\normalsize] at (7.75,9.5) {$\mathrm{cor}^2$};
\node [font=\normalsize] at (2,1.75) {$\mathrm{res}^3$};
\node [font=\normalsize] at (7.5,1.75) {$\mathrm{cor}^3$};
\node [font=\normalsize] at (2,-6.75) {$\mathrm{res}^4$};
\node [font=\normalsize] at (7.5,-6.75) {$\mathrm{cor}^4$};
\node [font=\large] at (0.25,15.5) {$v$};
\node [font=\large] at (4.75,15.5) {$v'$};
\node [font=\LARGE] at (3.25,16) {};
\node [font=\LARGE] at (3.25,16) {};
\draw [->, >=Stealth] (0.75,15.5) -- (4.25,15.5);
\node [font=\large] at (0.25,15) {$u$};
\node [font=\large] at (4.75,15) {$0$};
\draw [->, >=Stealth] (0.75,15) -- (4.25,15);
\draw [->, >=Stealth] (5.25,15.5) -- (9.75,15.5);
\node [font=\large] at (10.45,15) {$u+v$};
\node [font=\large] at (10.25,15.5) {$0$};
\node [font=\large] at (6,15) {$u'$};
\draw [->, >=Stealth] (6.75,15) -- (9.75,15);
\node [font=\LARGE] at (3,8.75) {};
\node [font=\large] at (0,8.25) {$u^2$};
\node [font=\large] at (4.75,8.25) {$0$};
\node [font=\LARGE] at (3,8.75) {};
\node [font=\LARGE] at (3.25,8.75) {};
\draw [->, >=Stealth] (0.75,8.25) -- (4.25,8.25);
\node [font=\large] at (0,7.75) {$v^2$};
\node [font=\large] at (4.75,7.75) {$(v')^2$};
\draw [->, >=Stealth] (0.75,7.75) -- (4.25,7.75);
\draw [->, >=Stealth] (5.25,7.75) -- (9.75,7.75);
\node [font=\large] at (10,8.25) {$?$};
\node [font=\large] at (10,7.75) {$0$};
\node [font=\large] at (5.5,8.25) {$w'$};
\draw [->, >=Stealth] (5.75,8.25) -- (9.25,8.25);
\node [font=\large] at (-0.5,7.25) {$w$};
\node [font=\large] at (4.5,7.25) {$\mathrm{res}(w)$};
\draw [->, >=Stealth] (0,7.25) -- (3.5,7.25);
\node [font=\large] at (10,7.25) {$\ast$};
\node [font=\large] at (6,7.25) {$(u')^2$};
\draw [->, >=Stealth] (6.5,7.25) -- (9.5,7.25);
\node [font=\large] at (-0,0.5) {$u^3$};
\node [font=\large] at (5,0.5) {$0$};
\draw [->, >=Stealth] (0.5,0.5) -- (4.75,0.5);
\node [font=\large] at (-0,0) {$v^3$};
\node [font=\large] at (5.25,0) {$(v')^3$};
\draw [->, >=Stealth] (0.5,0) -- (4.75,0);
\draw [->, >=Stealth] (5.75,0) -- (10.25,0);
\node [font=\large] at (10.5,0.5) {$\ast$};
\node [font=\large] at (10.5,0) {$0$};
\node [font=\large] at (6,0.5) {$(u')^3$};
\draw [->, >=Stealth] (6.25,0.5) -- (9.75,0.5);
\node [font=\large] at (0,-0.5) {$wv$};
\node [font=\large] at (5,-0.5) {$v'\mathrm{res}(w)$};
\draw [->, >=Stealth] (0.5,-0.5) -- (4,-0.5);
\draw [->, >=Stealth] (6,-0.5) -- (10.25,-0.5);
\node [font=\large] at (0,-1) {$wu$};
\draw [->, >=Stealth] (0.5,-1) -- (4,-1);
\node [font=\large] at (4.25,-1) {$0$};
\node [font=\large] at (10.5,-0.5) {0};
\draw [->, >=Stealth] (6,-1) -- (10.25,-1);
\node [font=\large] at (10.5,-1) {0};
\node [font=\large] at (5.5,-1.5) {$u'w'$};
\node [font=\large] at (5.5,-1) {$v'w'$};
\draw [->, >=Stealth] (6,-1.5) -- (10.25,-1.5);
\node [font=\large] at (11.25,-1.5) {$uw+vw$};
\node [font=\large] at (-3.5,-8) {$u^3$};
\draw [->, >=Stealth] (-3,-8) -- (0,-8);
\node [font=\large] at (-3.5,-8.5) {$uw$};
\draw [->, >=Stealth] (-3,-8.5) -- (0,-8.5);
\node [font=\large] at (-3.75,-9) {$u^3+v^3$};
\draw [->, >=Stealth] (-3,-9) -- (0,-9);
\node [font=\large] at (0.25,-8) {$u^4$};
\node [font=\large] at (0.5,-8.5) {$u^2w$};
\node [font=\large] at (0.5,-9) {$0$};

\node [font=\large] at (-3.75,-9.5) {$v^3$};
\draw [->, >=Stealth] (-3,-9.5) -- (0,-9.5);
\node [font=\large] at (0.5,-9.5) {$u^4$};
\node [font=\large] at (-3.5,-10) {$vw$};
\draw [->, >=Stealth] (-3,-10) -- (0,-10);
\node [font=\large] at (0.5,-10) {$u^2w$};
\node [font=\large] at (-4,-10.5) {$uw+vw$};
\draw [->, >=Stealth] (-3,-10.5) -- (0,-10.5);
\node [font=\large] at (0.5,-10.5) {$0$};
\node [font=\large] at (7.25,12.75) {$\cup u$};
\node [font=\large] at (6.5,4.75) {$\cup u$};
\node [font=\large] at (6.5,-3.25) {$\cup u$};
      \end{circuitikz}
    }%
  \end{center}
    \caption{GLES $vw+uw$ is negligible in $D_{2^n}$.}
    \label{GysinD2n-1}
\end{figure}
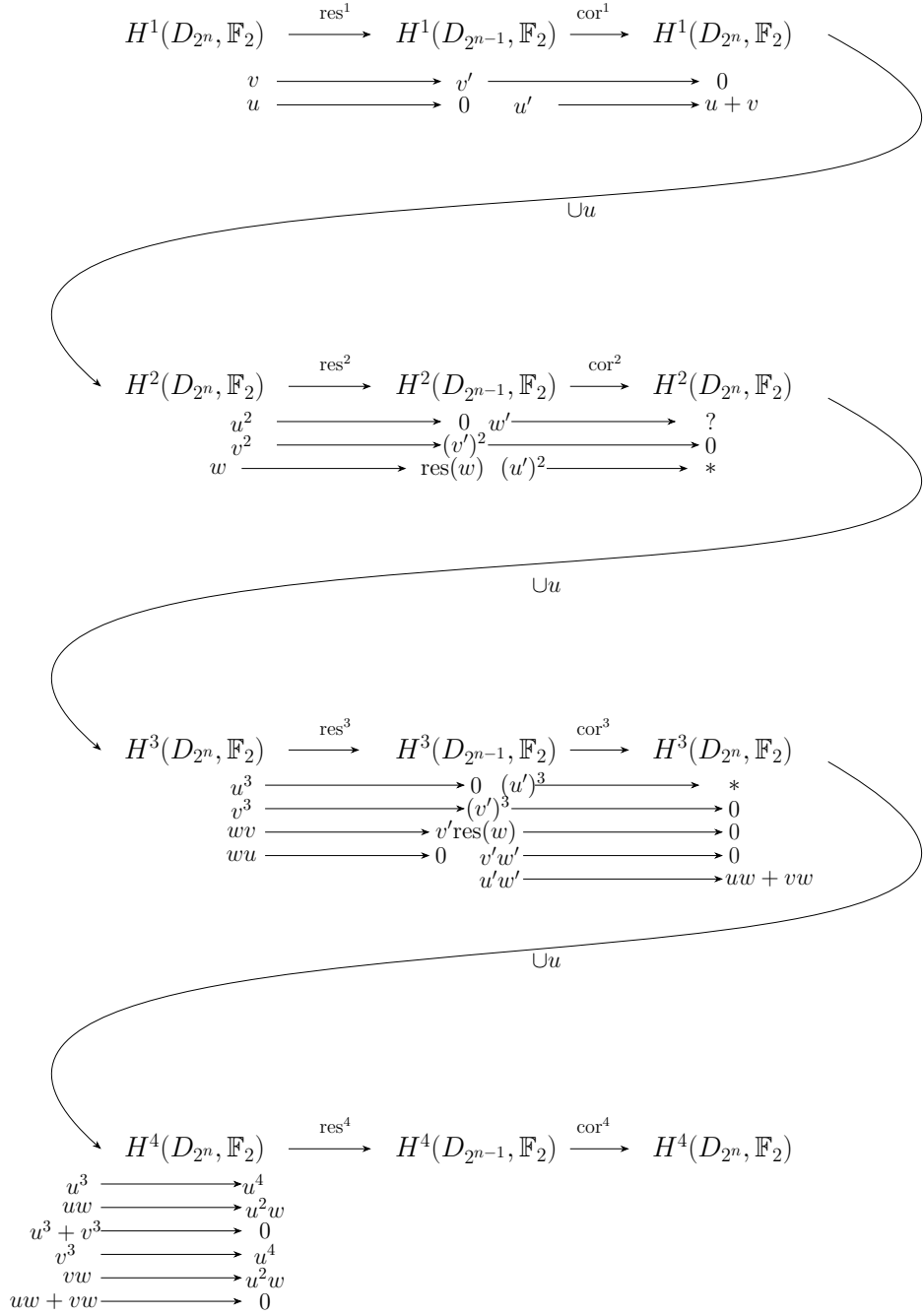

\subsection{Negligible ideal of the mod-2 cohomology of cyclic \texorpdfstring{$2$}{2}-groups}\label{computation-cyclic-2-groups}

In this section, we apply our computations of the negligible elements in $D_{2^n}$ (Section \ref{dihedral}), to derive the negligible elements of cyclic 2-groups via the Gysin long exact sequence.
Specifically, we leverage the fact that $D_{2^{n + 1}}$ contains a cyclic subgroup of order $2^n$.

Recall the cohomology ring of $\Z / 2^n \Z, n \geq 2$ can be expressed as 
\begin{equation*}
    H^*(\Z / 2^n \Z, \mathbb F_2) \cong \mathbb F_2 [x, y] / \langle x^2 \rangle,
\end{equation*}
with $\deg x = 1$ and $\deg y = 2$. We first compute the eventually negligible ideal.

\begin{prop}\label{cyclic-2-eventually-negligible}
For $n\geq 2$, $N(\Z/2^n\Z,\mathbb F_2)=\langle x\rangle$.
\end{prop}

\begin{proof}
The unique elementary abelian $2$-subgroup of $\Z/2^n\Z$ is $E\simeq \Z/2\Z$, with $H^*(E,\mathbb F_2)\cong\mathbb F_2[t]$ where $\deg t=1$.
The restriction map on generators is $$\res(x)=0,\qquad \res(y)=t^2.$$
Since $H^*(E,\mathbb F_2)$ is a polynomial ring on a single degree-one generator, Theorem~\ref{negideal} gives $I(E,\mathbb F_2)=0$, so no non-zero homogeneous class of $H^*(E,\mathbb F_2)$ is negligible. By Theorem~\ref{detect on elem ab}, a class of $H^*(\Z/2^n\Z,\mathbb F_2)$ is eventually negligible if and only if it restricts to zero on $E$. Since $\res(x)=0$ while $\res(y^r)=t^{2r}\ne 0$ for $r\geq 1$, the kernel of $\res$ is exactly $\langle x\rangle$, hence $N(\Z/2^n\Z,\mathbb F_2)=\langle x\rangle$.
\end{proof}

\begin{prop}\label{cyclic-2-xy-negligible}
The class $xy$ is negligible in $H^*(\Z/{2^n}\Z,\F_2)$. 
\end{prop}

\begin{proof}
Let $D_{2^{n+1}}=\langle r,s\rangle$ be the dihedral group containing the cyclic subgroup $\langle r\rangle\simeq \Z/2^n\Z$.
Use the standard generators $u,v,w\in H^*(D_{2^{n+1}},\mathbb F_2)$, and choose them so that $\langle r\rangle=\ker(v)$. Let $x\in H^1(\langle r\rangle,\mathbb F_2)$ and $y\in H^2(\langle r\rangle,\mathbb F_2)$ be the standard generators of
\[
H^*(\langle r\rangle,\mathbb F_2)\cong \mathbb F_2[x,y]/\langle x^2\rangle.
\]
On degree-one generators, restriction gives
\[
\res(u)=x,\qquad \res(v)=0.
\]
The degree-two lift in the same restriction computation gives $\res(w)=y$, hence
\[
\res(uw)=xy.
\]
By Proposition \ref{prop:uw-and-uw+vw-negligible}, the class $uw$ is negligible in $H^*(D_{2^{n+1}},\mathbb F_2)$. Lemma \ref{half conjecture} then implies that its restriction $xy$ is negligible in $H^*(\langle r\rangle,\mathbb F_2)$.
\end{proof}

\begin{cor}\label{cyclic-2-negligible}
For $n\geq 2$, $I(\Z/2^n\Z,\mathbb F_2)=\langle xy\rangle$.
\end{cor}

\begin{proof}
Every negligible class is eventually negligible, so by Proposition~\ref{cyclic-2-eventually-negligible},
\[
I(\Z/2^n\Z,\mathbb F_2)\subseteq N(\Z/2^n\Z,\mathbb F_2)=\langle x\rangle.
\]
Using $x^2=0$, the homogeneous elements of $\langle x\rangle$ are exactly the classes $\alpha\, xy^k$ with $\alpha\in\mathbb F_2$ and $k\geq 0$, of degree $2k+1$. Let $z\in I(\Z/2^n\Z,\mathbb F_2)$ be a non-zero homogeneous element, so $z=xy^k$ for some $k\geq 0$. The case $k=0$ would give the degree-one class $z=x$, which is not negligible by Theorem~\ref{nonneg-low-degree}. Hence $k\geq 1$ and $z=xy\cdot y^{k-1}\in\langle xy\rangle$. This proves $I(\Z/2^n\Z,\mathbb F_2)\subseteq\langle xy\rangle$. The reverse inclusion follows from Proposition~\ref{cyclic-2-xy-negligible}.
\end{proof}

\subsection{Eventually negligible ideal of mod-2 cohomology of finite abelian 2-groups}\label{compputation-finite-abelian-2}

In this section, we consider an arbitrary finite abelian 2-group 
$$G\cong (\mathbb Z/2\mathbb Z)^{\ell} \oplus \left(\bigoplus_{j = \ell + 1}^m \mathbb Z / 2^{k_j} \mathbb Z\right), \qquad k_j \ge 2$$
such that there are $\ell$ total $\mathbb Z/2\mathbb Z$ factors and $m - \ell$ higher order factors $(k_j \geq 2)$, giving a total of $m$ factors.
We do not consider elementary abelian or cyclic groups here since the classification for each of these cases is handled in \cite{GhermanMerkurjevJPAA} and Section \ref{computation-cyclic-2-groups} respectively.
In other words, we assume that $m> \ell \geq 0$ and $m\geq 2$ throughout the section.
The group $G$ contains a unique maximal elementary abelian subgroup isomorphic to $(\mathbb Z/2\mathbb Z)^m$, which we denote $E_{\max}$.

Our computations are greatly simplified upon noting the following fact about the cohomology of cyclic 2-groups.
\begin{prop}[cf. Section 3.2 and 3.5 of \cite{Evens}]
For $\ell_i \geq 2$,
\[H^*\!\left(\mathbb{Z}/2^{\ell_i}\mathbb{Z},\,\mathbb F_2\right)\cong\mathbb F_2[x_i,y_i]/\langle x_i^2\rangle.\]
\end{prop}

We construct the cohomology of $G$ using K\"unneth formula.
With $\mathbb F_2$-coefficients, we simply adjoin the generators and relations from each factor of $G$.
Each $\mathbb Z/2\mathbb Z$ factor contributes a single degree 1 generator, which we denote $x_i$.
The restriction map from the cohomology of $G$ to that of $E_{\max}$ is the identity $x_{i} \mapsto x_{i}$.

Each $\mathbb Z / 2^{k_j} \mathbb Z$ factor for $k_j\geq 2$ contributes a degree 1 generator, denoted $x_{j}$, and a degree 2 generator, denoted $y_{j}$. The restriction map is
\[ \res(x_j)=0, \qquad \res(y_j) = x_j^2,\]
and there is also a single ring relation $x_j^2$.  For a general finite abelian 2-group as above, the generators would be
\begin{equation}\label{elem ab indices}
    \{x_1, \dots, x_\ell, x_{\ell + 1}, \dots, x_m, y_{\ell  + 1}, \dots y_{m}\}
\end{equation}
with $\deg(x_k)=1$ for all $1\leq k\leq m$ and $\deg(y_j)=2$ for all $\ell+1\leq j\leq m$.
The restriction maps from the mod-2 cohomology of $G$ to that of $E_{\max}$ are
\begin{align*}
     &\text{res}(x_i) = x_i \quad \text{for } 1 \leq i \leq \ell, 
    \\ &\text{res}(x_{j}) = 0 \quad \text{for } \ell + 1 \leq j \leq m,
    \\ &\text{res}(y_{j}) = x_{j}^2 \quad  \text{ for } \ell + 1 \leq j \leq m,
\end{align*}
and recall also that each $x_j$ coming from a higher-order cyclic factor satisfies the ring relation $x_j^2=0$ in $H^*(G,\mathbb F_2)$. Note that in the formula $\res(y_j)=x_j^2$ above, the symbol $x_j$ on the right-hand side denotes instead the polynomial generator of $H^*(E,\mathbb F_2)$, where no such relation holds.
Now we can identify our candidate generators for the eventually negligible ideal.
\begin{prop}\label{2x4}
    
    With indices as in Equation (\ref{elem ab indices}), and $1\leq i<i'\leq \ell$ and  $\ell+1\leq j<j'\leq m$,
    \begin{equation*}
        \langle  x_j, x_i^2x_{i'} + x_ix_{i'}^2, y_j^2y_{j'} + y_jy_{j'}^2, x_iy_j(x_i^2 + y_j)\rangle \subseteq N(G,\mathbb F_2).
    \end{equation*}
    
\end{prop}

\begin{proof}
The class $x_j$ restricts to zero on $E_{\max}$. The class
$x_i^2x_{i'}+x_ix_{i'}^2$ restricts to a generator of
$I(E_{\max},\mathbb F_2)$ and is therefore eventually negligible. The third generator restricts to
\[
x_j^2x_{j'}^4+x_{j'}^2x_j^4=(x_jx_{j'}^2+x_j^2x_{j'})^2,
\]
and the fourth restricts to $$x_i x_j^2 (x_i^2+x_j^2)=(x_ix_j^2+x_i^2x_j)\cdot x_j(x_i+x_j).$$
By Theorem~\ref{negideal}, both of these images lie in $I(E,\mathbb F_2)$, so by Theorem ~\ref{detect on elem ab}, the generators are eventually negligible.
\end{proof}

Since we know the elements in Propositions \ref{2x4} map to zero in the quotient of the cohomology ring by the eventually negligible elements, we can rearrange to get algebraically equivalent expressions of certain elements in a more desirable form.
Namely, if we have a term involving two monomials raised to powers, we may rewrite it so that one of the monomials is raised to the power of one.

\begin{cor}
\label{lem2x4}
    Let $G$ be a finite abelian $2$-group, $x_i$ be a degree one generator of $H^*(G,\mathbb F_2)$, and $y_j$ a degree two generator of $H^*(G,\mathbb F_2)$.
    Then for all $k\geq 3$ and $\ell\geq1$, we have 
    $$x_i^{k}y_j^{\ell} = x_i^{k+2(\ell-1)}y_j$$ in the quotient $H^*(G,\mathbb F_2)/N(G,\mathbb F_2)$.
\end{cor}
\begin{proof}
    Since $x_i y_j(x_i^2 + y_j) \in N(G, \mathbb F_2)$ by Proposition \ref{2x4}, this element must map to zero in the quotient.
    With a slight abuse of notation, we can say in the quotient
    $$x_i y_j(x_i^2 + y_j) = x_i^3y_j + x_i y_j^2 = 0.$$
    Since the coefficients lie in $\mathbb F_2$, this gives  $$x_i^3y_j=x_iy_j^2.$$
    Now write 
    \[
    x_i^k y_j^{\ell}
    =x_i^{k-1}x_i y_j^{\ell-2}y_j^2
    =x_i^{k-1}y_j^{\ell-2}x_i^{3}y_j
    =x_i^{k+2}y_j^{\ell-1}.
    \]
    By induction, we obtain 
    \[
    x_i^k y_j^{\ell}
    = x_i^{k+2(\ell-1)}y_j. \qedhere
    \]
\end{proof}
\begin{cor}
\label{lem32}
    Let $G$ be a finite abelian 2-group, and $x_i$, $x_j$ be distinct degree one generators of $H^*(G,\mathbb F_2)$. For all $k \geq 2$ and $\ell \geq 1$,
    \begin{equation*}
        x_i^k x_j^\ell = x_i^{k+\ell - 1}x_j
    \end{equation*}
    in $H^*(G,\mathbb F_2)/N(G,\mathbb F_2)$.
\end{cor}
\begin{proof}
    The restriction maps are either the identity or the zero map. Thus, by Theorem \ref{negideal}, $x_i^2x_j + x_ix_j^2$ are in the eventually negligible ideal, and
    $$ x_i^2 x_j + x_ix_j^2 = 0\implies  x_i^2x_j = x_ix_j^2\mod 2.$$
    Then observe
    $$x_i^kx_j^\ell=x_i^{k - 1}x_j^{\ell - 2} (x_i x_j^2) = x_i^{k - 1}x_j^{\ell - 2} (x_i^2 x_j) = x_i^{k + 1}x_j^{\ell - 1}.$$
    The result follows by induction.
\end{proof}

\begin{theorem}
    Let $G$ be a finite abelian 2-group that is neither cyclic nor elementary abelian.
    With indices as in Equation (\ref{elem ab indices}),
    $$N(G,\mathbb F_2)  = \langle x_j, x_i^2x_{i'} + x_ix_{i'}^2, y_j^2y_{j'} + y_jy_{j'}^2, x_iy_j(x_i^2 + y_j)\rangle.$$
\end{theorem}

\begin{proof}
Let
$$ J=\langle x_j,\ x_i^2x_{i'}+x_ix_{i'}^2,\ y_j^2y_{j'}+y_jy_{j'}^2,\ x_iy_j(x_i^2+y_j)\rangle,$$
where the indices have the same ranges as in the statement. Proposition~\ref{2x4} gives $J\subseteq N(G,\mathbb F_2)$, so it remains to prove the reverse inclusion. Since $E_{\max}$ is the unique maximal elementary abelian $2$-subgroup of $G$, every elementary abelian $2$-subgroup is contained in $E_{\max}$. Hence, by Corollary~\ref{quillen-venkov} and Lemma~\ref{half conjecture}, $N(G,\mathbb F_2)=\operatorname{res}^{-1}(I(E_{\max},\mathbb F_2))$. It is therefore enough to prove that the induced map
$$\overline{\operatorname{res}}\colon H^*(G,\mathbb F_2)/J
\longrightarrow H^*(E_{\max},\mathbb F_2)/I(E_{\max},\mathbb F_2)$$
is injective. Write $L=\{1,\dots,\ell\}$ and $H=\{\ell+1,\dots,m\}$, and argue degree by degree. The degree-zero case is immediate, so fix $d>0$. In the quotient by $J$, the generators $x_j$ for $j\in H$ vanish. Thus every degree-$d$ monomial is represented by $M=\prod_{i\in S_x}x_i^{a_i}\prod_{j\in S_y}y_j^{b_j}$, where $S_x\subseteq L$, $S_y\subseteq H$, and $a_i,b_j\geq1$. The defining relations of $J$ imply, for positive integers $a,b$,
$$ x_i^ax_{i'}^b=x_i^{a+b-1}x_{i'},\qquad
y_j^ay_{j'}^b=y_j^{a+b-1}y_{j'},\qquad
x_i^ay_j^b=x_i^{a+2(b-1)}y_j. $$

Suppose first that $S_x\neq\varnothing$, and let $i_0$ be its smallest element. Repeatedly applying these relations reduces $M$ to the canonical monomial $M_{S_x,S_y}^{(d)}=x_{i_0}^{\,d-|S_x|-2|S_y|+1}\prod_{i\in S_x\setminus\{i_0\}}x_i\prod_{j\in S_y}y_j$. The exponent of $x_{i_0}$ is positive because $d\geq |S_x|+2|S_y|$. If $S_x=\varnothing$, then $S_y\neq\varnothing$ and $d$ is even. Writing $j_0$ for the smallest element of $S_y$, the $y$-relations reduce $M$ to $M_{\varnothing,S_y}^{(d)}=y_{j_0}^{\,d/2-|S_y|+1}\prod_{j\in S_y\setminus\{j_0\}}y_j$. Consequently, every homogeneous $f\in H^d(G,\mathbb F_2)/J$ can be written as $f=\sum_{S_x,S_y}c_{S_x,S_y}M_{S_x,S_y}^{(d)}$, where $c_{S_x,S_y}\in\mathbb F_2$ and each pair $(S_x,S_y)$ occurs at most once. Indeed, monomials with the same support reduce to the same canonical monomial, after which their coefficients are combined in $\mathbb F_2$. Let $f\neq0$, and choose a pair $(S_x,S_y)$ with $c_{S_x,S_y}=1$ such that $S=S_x\cup S_y$ is minimal under inclusion among the supports with non-zero coefficient. Choose polynomial generators $t_1,\dots,t_m$ for $H^*(E_{\max},\mathbb F_2)=\mathbb F_2[t_1,\dots,t_m]$ so that $\operatorname{res}(x_i)=t_i$ for $i\in L$ and $\operatorname{res}(y_j)=t_j^2$ for $j\in H$. Let $\mathbf 1_S\in\mathbb F_2^m$ have $r$-th coordinate equal to $1$ when $r\in S$ and $0$ otherwise. A canonical monomial with support $T$ evaluates at $\mathbf 1_S$ to $1$ exactly when $T\subseteq S$. By minimality, no proper subset of $S$ occurs with non-zero coefficient, and by construction there is only one canonical monomial with support $S$. Thus $\operatorname{res}(f)(\mathbf 1_S)=1$. By Proposition~\ref{elem ab 2-group vanishing}, $\operatorname{res}(f)$ is not negligible, so $\overline{\operatorname{res}}(f)\neq0$. Therefore $\overline{\operatorname{res}}$ is injective and $N(G,\mathbb F_2)\subseteq J$. Together with Proposition~\ref{2x4}, this proves $N(G,\mathbb F_2)=J$.
\end{proof}

\subsection{Negligible ideal of the mod-2 cohomology of generalized quaternions}\label{computation-generalized-quaternions}

In this section, we leverage the Gysin sequence technique to find the negligible ideal of $Q_{2^n}$ for $n \geq 4$.
The $n = 3$ case in Proposition \ref{Q8} was recently proved by Baek and Merkurjev in \cite[Theorem 7.1]{BaekMerkurjev}.
The mod-2 cohomology ring for $Q_8$ is 
\begin{equation*}
    H^*(Q_8, \mathbb F_2) \cong \mathbb F_2[a, b, c] / \langle a^2 + ab + b^2, a^3 \rangle,
\end{equation*}
where $\text{deg}(a) = \text{deg}(b) = 1$ and $\text{deg}(c) = 4$.

\begin{prop}[cf. Theorem 7.1 of \cite{BaekMerkurjev}]\label{Q8}
    The negligible ideal of $H^*(Q_{8}, \mathbb F_2)$ is
    \begin{equation*}
        I(Q_8,\F_2)=\langle a^2b, ac, bc\rangle.
    \end{equation*}
\end{prop}

The cohomology of $Q_{2^n}$ is identical for all $n \geq 4$, yet differs from $Q_8$. Since the Gysin long exact sequence we employ is based on the inclusion $Q_8 \subset Q_{16}$, 
we first apply it to compute the negligible ideal of $H^*(Q_{16}, \mathbb F_2)$, which serves as the base case for our induction. We then proceed inductively on $n$, restricting to the index-$2$ subgroups $Q_{2^{n-1}}$ and $\Z / {2^{n-1}} \Z$ of $Q_{2^n}$ to determine the negligible ideal of $H^*(Q_{2^n}, \mathbb F_2)$.

The mod-2 cohomology ring of $Q_{2^n}$ for $n \geq 4$ can be expressed as 
$$H^*(Q_{2^n},\mathbb F_2)\cong \mathbb F_2[a, b, c]/\langle ab, a^3 + b^3\rangle,$$
where $\text{deg}(a) = \text{deg}(b) = 1$ and $\text{deg}(c) = 4$.
Similarly, in the inductive step, the mod-2 cohomology ring of $Q_{2^{n - 1}}$ for $n \geq 5$ is expressed as 
$$H^*(Q_{2^{n - 1}},\mathbb F_2)\cong \mathbb F_2[a', b', c']/\langle a'b', a'^3 + b'^3\rangle.$$
When we restrict to $Q_8$, we write its cohomology ring as
$$H^*(Q_8, \mathbb F_2) \cong \mathbb F_2[a', b', c'] / \langle a'^2b' + a'b'^2, a'^3 \rangle.$$
We compute the restriction maps by considering the degree 1 case. Recall 
\begin{align*}
H^1(Q_{2^n};\mathbb Z/2\mathbb Z)
&= \operatorname{Hom}(Q_{2^n},\mathbb Z/2\mathbb Z) \\
&\cong \operatorname{Hom}\bigl(Q_{2^n}/[Q_{2^n},Q_{2^n}],\mathbb Z/2\mathbb Z\bigr) \\
&\cong \operatorname{Hom}\bigl(\mathbb Z/2\mathbb Z \times \mathbb Z/2\mathbb Z,\mathbb Z/2\mathbb Z\bigr) \\
&\cong \mathbb Z/2\mathbb Z \times \mathbb Z/2\mathbb Z .
\end{align*}

The action of the generators $a, b$ of $\text{Hom} (Q_{2^n}, \mathbb F_2)$ on the generators $i$ and $j$ of the quaternion group can be expressed as 
\[
\begin{matrix}
  a(i) = \overline{1},  &
  a(j) = \overline{0},\\
  b(i) = \overline{1},  &
  b(j) =\overline{1}.
\end{matrix}
\] 
Additionally, the generators of $\text{Hom} (Q_{2^{n - 1}}, \mathbb F_2)$ can be defined the same way, using $a'$ and $b'$. There are three index 2 subgroups of $Q_{2^n}$. Two of the three are isomorphic to $Q_{2^{n - 1}}$, and one of them is isomorphic to $C_{2^{n - 1}}$
\[ 
\langle i^2, ij \rangle \cong Q_{2^{n - 1}},\qquad \langle i^2, j \rangle \cong Q_{2^{n - 1}}, \qquad \langle i \rangle \cong C_{2^{n - 1}}. 
\]

For the first $Q_{2^{n - 1}}$ subgroup, by restricting $a$ and $b$ to the generators of the subgroup, we can deduce,
\[\text{res}(a)= a' + b', \qquad \text{res}(b) = 0.\]
Similarly, for the second $Q_{2^{n - 1}}$ subgroup, we can deduce
\[  \text{res}(a)= 0,\qquad \text{res}(b) = a' + b'.\]
For the cyclic subgroup, we can deduce
\[ \text{res}(a)= x , \qquad \text{res}(b) = x. \]

\begin{prop}
    The negligible ideal of $H^*(Q_{2^n}, \mathbb F_2)$ for $n \geq 4$ is
    \begin{equation*}
        I(Q_{2^n},\F_2)=\langle a^3, ac, bc\rangle.
    \end{equation*}
\end{prop}

\begin{proof}
The proof is obtained by chasing the Gysin long exact sequences in Figures~\ref{Q16resC8}--\ref{Q32resQ16-2}.
In each case the chosen index-two subgroup is the kernel of the indicated degree-one character, so the connecting maps are cup product with that character, and exactness identifies the kernel of this cup-product map with the image of the corresponding corestriction map.
Figure~\ref{Q16resC8} shows that, in the base case $Q_{16}$, the negligible class $xy\in H^3(C_8,\mathbb F_2)$ corestricts to $a^3=b^3$, while Figures~\ref{Q16resQ8-1} and~\ref{Q16resQ8-2} show that $bc$ and $ac$ are likewise obtained by corestricting negligible classes from the relevant index-two subgroups.
Lemma~\ref{lemma:corestriction} therefore gives $a^3,ac,bc\in I(Q_{16},\mathbb F_2)$.
The same diagram chase in Figures~\ref{Q32resC16}--\ref{Q32resQ16-2} proves the inductive step.
Figure~\ref{Q32resC16} gives $a^3=b^3$, and Figures~\ref{Q32resQ16-1} and~\ref{Q32resQ16-2} give $ac$ and $bc$.
Hence $\langle a^3,ac,bc\rangle\subseteq I(Q_{2^n},\mathbb F_2)$.

For the reverse inclusion, we first compute the eventually negligible ideal. We claim that
$N(Q_{2^n},\mathbb F_2)=\langle a,b\rangle$ for $n\geq 4$. Recall that
$$ H^*(Q_{2^n},\mathbb F_2)
= \mathbb F_2[a,b,c]/\langle ab,\ a^3+b^3\rangle, \quad
\deg(a)=\deg(b)=1,\quad \deg(c)=4. $$
The group $Q_{2^n}$ has a unique subgroup of order $2$, namely
$C=\langle z\rangle\cong C_2$, where $z=i^{2^{n-2}}$.
Thus, $C$ is the only non-trivial elementary abelian $2$-subgroup of $Q_{2^n}$.
By Corollary~\ref{quillen-venkov}, a class $u\in H^*(Q_{2^n},\mathbb F_2)$ is eventually negligible if and only if $\mathrm{res}_C^{Q_{2^n}}(u)$ is negligible.
Since $H^*(C,\mathbb F_2)\cong \mathbb F_2[t]$, with $\deg(t)=1$, and $I(C,\mathbb F_2)=0$ by Theorem~\ref{negideal}, this gives
$$ u\in N(Q_{2^n},\mathbb F_2) \iff\mathrm{res}_C^{Q_{2^n}}(u)=0.$$
Hence, $N(Q_{2^n},\mathbb F_2)=\ker(\mathrm{res}_C^{Q_{2^n}})$. We now compute this kernel.
Since $a$ and $b$ are degree-one classes, they are homomorphisms $Q_{2^n}\to \mathbb F_2$.
The element $z=i^{2^{n-2}}$ is an even power of $i$, so every degree-one class vanishes on $z$.
Hence, $\mathrm{res}_C^{Q_{2^n}}(a)=\mathrm{res}_C^{Q_{2^n}}(b)=0$, so $\langle a,b\rangle\subseteq \ker(\mathrm{res}_C^{Q_{2^n}})$.
Further, since $\deg(c)=4$, we have $\res_C^{Q_{2^n}}(c)\in H^4(C,\mathbb F_2)=\mathbb F_2\cdot t^4$, so $\mathrm{res}_C^{Q_{2^n}}(c)$ is either $0$ or $t^4$.
If it were $0$, then $c$ would restrict trivially to every elementary abelian $2$-subgroup of
$Q_{2^n}$, so Theorem~\ref{thm:quillen-venkov-nilpotence} would imply
that $c$ is nilpotent. But
$$ H^*(Q_{2^n},\mathbb F_2)/\langle a,b\rangle\cong \mathbb F_2[c],$$
so no positive power of $c$ is zero. Thus $\mathrm{res}_C^{Q_{2^n}}(c)\neq 0$,
and hence $\mathrm{res}_C^{Q_{2^n}}(c)=t^4$. After quotienting by
$\langle a,b\rangle$, the induced map is
$$
\mathbb F_2[c]\to \mathbb F_2[t],
\qquad
c\mapsto t^4,
$$
which is injective. Therefore
$\ker(\mathrm{res}_C^{Q_{2^n}})=\langle a,b\rangle$, and hence
$$
N(Q_{2^n},\mathbb F_2)=\langle a,b\rangle.
$$

Let $u\in I(Q_{2^n},\mathbb F_2)$ be homogeneous.
Since every negligible class is eventually negligible, $u\in N(Q_{2^n},\mathbb F_2)=\langle a,b\rangle$.
If $\deg(u)=1$ or $\deg(u)=2$, then $u=0$ by Theorem~\ref{nonneg-low-degree}.
Hence $u\in\langle a^3,ac,bc\rangle$ in
degrees $0,1,$ and $2$.
Assume now that $\deg(u)\geq 3$.
In $H^*(Q_{2^n},\mathbb F_2)$ we have
$a^3=b^3$, and hence
$$ a^4=a(a^3)=a(b^3)=(ab)b^2=0, \qquad b^4=b(b^3)=b(a^3)=(ab)a^2=0. $$
Every homogeneous element of $\langle a,b\rangle$ is a sum of monomials which contain a factor of $a$ or $b$.
If such a monomial has degree at
least $3$ and contains no factor of $c$, then the relations above imply that it is a multiple of $a^3=b^3$.
If it contains a factor of $c$, then, since it also contains a factor of $a$ or $b$, it is a multiple of $ac$ or $bc$.
Therefore,
$$ \langle a,b\rangle\cap H^{\geq 3}(Q_{2^n},\mathbb F_2)
\subseteq \langle a^3,ac,bc\rangle, $$
and $u\in\langle a^3,ac,bc\rangle$.
Since every homogeneous negligible class lies in $\langle a^3,ac,bc\rangle$ and the negligible ideal is graded,
we obtain
$$ I(Q_{2^n},\mathbb F_2)\subseteq \langle  a^3,ac,bc\rangle,$$
proving the reverse inclusion.
\end{proof}

\section{Radicality for finite irreducible Coxeter groups}\label{conjecture-sec-coxeter}

We now complete the proof of Theorem~\ref{conjecture}.

\subsection{Dihedral Coxeter groups}
In Section \ref{dihedral}, we proved that all dihedral groups have radical mod-$2$ negligible ideal over $\Q$.
It remains to treat the exceptional groups $H_3$ and $H_4$.

\subsection{Exceptional group \texorpdfstring{$H_3$}{H3}} \label{h3}
\begin{rem}
In what follows, we use the notation $H_3$ and $H_4$ for the finite Coxeter
groups of types $H_3$ and $H_4$, respectively. Strictly speaking, $H_3$ and
$H_4$ denote Coxeter types or Coxeter diagrams, and the corresponding finite
Coxeter groups are generated by the simple reflections attached to those
diagrams. We follow the common convention of writing $H_3$ and $H_4$ for the
groups themselves.
\end{rem}
\begin{lemma}\label{radicality-h3}
Let $H_3$ denote the finite Coxeter group of type $H_3$. Then the
mod-$2$ negligible ideal $I(H_3,\mathbb F_2)$ over $\mathbb Q$ is radical.
\end{lemma}

\begin{proof}

The Coxeter group of type $H_3$ is isomorphic to $ A_5\times \mathbb Z/2\Z$ \cite[Section~6]{KOCA2006441}. Let
$V\leq A_5$ be the subgroup $$ V=  \{1,(12)(34),(13)(24),(14)(23)\}.$$
Then $V$ is a Klein four group, so
$V\cong \mathbb Z/2\Z\times \mathbb Z/2\Z.$
Moreover, since $|A_5|=60$, the subgroup $V$ is a Sylow-$2$ subgroup of $A_5$. Choose an isomorphism $H_3\cong A_5\times \mathbb Z/2\Z,$ and let $E\leq H_3$ be the subgroup corresponding to $ V\times \mathbb Z/2\Z.$ Then $ E\cong(\mathbb Z/2\Z)^3 $ and $[H_3:E]=|H_3|/|E|= 120/8=15.$ In particular, $[H_3:E]$ is prime to $2$. Since $E$ is elementary abelian,
Theorem~\ref{negideal} and Proposition~\ref{elem ab 2-group vanishing} imply that $I(E,\mathbb F_2)$ is radical.
Now one can apply Lemma~\ref{intersection lemma} to the inclusion
$E\leq H_3$, and restriction gives an injection
$$ \mathrm{res}_{E}^{H_3}: H^\ast(H_3,\mathbb F_2) \hookrightarrow H^\ast(E,\mathbb F_2).$$
Identifying $H^\ast(H_3,\mathbb F_2)$ with its image under restriction, Lemma~\ref{intersection lemma} gives
$$I(H_3,\mathbb F_2) =H^\ast(H_3,\mathbb F_2)\cap I(E,\mathbb F_2).$$

Now let $x\in H^\ast(H_3,\mathbb F_2)$ and suppose that
$ x^m\in I(H_3,\mathbb F_2)$ for some $m\geq 1$. Restricting to $E$, we get $$ \mathrm{res}_{E}^{H_3}(x)^m  =\mathrm{res}_{E}^{H_3}(x^m) \in I(E,\mathbb F_2).$$
Since $I(E,\mathbb F_2)$ is radical, it follows that
$\mathrm{res}_{E}^{H_3}(x)\in I(E,\mathbb F_2). $
But $\mathrm{res}_{E}^{H_3}(x)$ also lies in the image of
$H^\ast(H_3,\mathbb F_2)$ inside $H^\ast(E,\mathbb F_2)$. Hence, by the intersection formula, $ x\in I(H_3,\mathbb F_2). $
Therefore $I(H_3,\mathbb F_2)$ is radical.
\end{proof}

\subsection{Exceptional group \texorpdfstring{$H_4$}{H4}}\label{h4prod}

\begin{lemma}\label{radicality-h4}
The mod-$2$ negligible ideal $I(H_4,\mathbb F_2)$ over $\mathbb Q$ is radical.
\end{lemma}

\begin{proof}
By \cite[Section~5]{KOCA2006441}, the Coxeter group of type $H_4$ contains
a subgroup of the form
$
        W(D_4)\rtimes \mathbb Z/3\Z.
$
In particular, $H_4$ contains a subgroup isomorphic to $W(D_4)$, the Weyl
group of type $D_4$. Here $D_4$ denotes the Coxeter type and $W(D_4)$ its Weyl group. This
should not be confused with the notation used above for dihedral groups. The Weyl group $W(D_4)$ has order $ |W(D_4)|=192,$ and the Coxeter group $H_4$ has order $|H_4|=14400.$ Since $[H_4:W(D_4)] = 14400/192 =75,$ Lemma~\ref{intersection lemma} gives the inclusion
$ W(D_4)\leq H_4.$ Thus restriction gives an injection
$$\mathrm{res}_{W(D_4)}^{H_4}: H^\ast(H_4,\mathbb F_2)
\hookrightarrow H^\ast(W(D_4),\mathbb F_2),$$
and, after identifying $H^\ast(H_4,\mathbb F_2)$ with its image under restriction, $$ I(H_4,\mathbb F_2)= H^\ast(H_4,\mathbb F_2)\cap I(W(D_4),\mathbb F_2).$$
Since $W(D_4)$ is a Weyl group, Theorem~\ref{Weylneg} implies that
$I(W(D_4),\mathbb F_2)$ is radical. Now let $x\in H^\ast(H_4,\mathbb F_2)$ and suppose that $x^m\in I(H_4,\mathbb F_2)$ for some $m\geq 1$. Restricting to $W(D_4)$,
we get $$ \mathrm{res}_{W(D_4)}^{H_4}(x)^m =\mathrm{res}_{W(D_4)}^{H_4}(x^m)\in I(W(D_4),\mathbb F_2).$$
Since $I(W(D_4),\mathbb F_2)$ is radical, it follows that
$\mathrm{res}_{W(D_4)}^{H_4}(x)\in I(W(D_4),\mathbb F_2).$
But $\mathrm{res}_{W(D_4)}^{H_4}(x)$ also lies in the image of
$H^\ast(H_4,\mathbb F_2)$ inside $H^\ast(W(D_4),\mathbb F_2)$. Hence, by
the intersection formula,
$x\in I(H_4,\mathbb F_2).$
Therefore $I(H_4,\mathbb F_2)$ is radical.
\end{proof}

\begin{proof}[Proof of Theorem~\ref{conjecture}]
By Theorem~\ref{Weylneg}, all Weyl groups have radical mod-$2$ negligible ideal over
$\mathbb Q$. By the classification of finite irreducible Coxeter groups, the
irreducible Coxeter groups which are not Weyl groups are the dihedral groups
$I_2(m)$ and the exceptional noncrystallographic groups $H_3$ and $H_4$.
The dihedral cases were proved in Section~\ref{dihedral}, the $H_3$ case was proved in
Lemma~\ref{radicality-h3}, and the $H_4$ case is Lemma~\ref{radicality-h4}. Hence the
result holds for all finite irreducible Coxeter groups.
\end{proof}

\begingroup
\raggedright
\sloppy

\endgroup

\clearpage
\appendix
\counterwithin{figure}{section}              
\renewcommand{\thefigure}{A\arabic{figure}}   
\setcounter{figure}{0}

\section{Gysin Long Exact Sequences (GLES)}\label{appendix}

\begingroup
  \let\clearpage\relax
  \renewcommand{\listfigurename}{}
  \listoffigures
\endgroup

\medskip
In figures \ref{Q16resQ8-1} and \ref{Q16resQ8-2}, the symbol $``\ast"$ denotes arbitrary elements of the indicated cohomology group. Their precise value plays no role in the computation; it is simply a placeholder indicating an unspecified class.

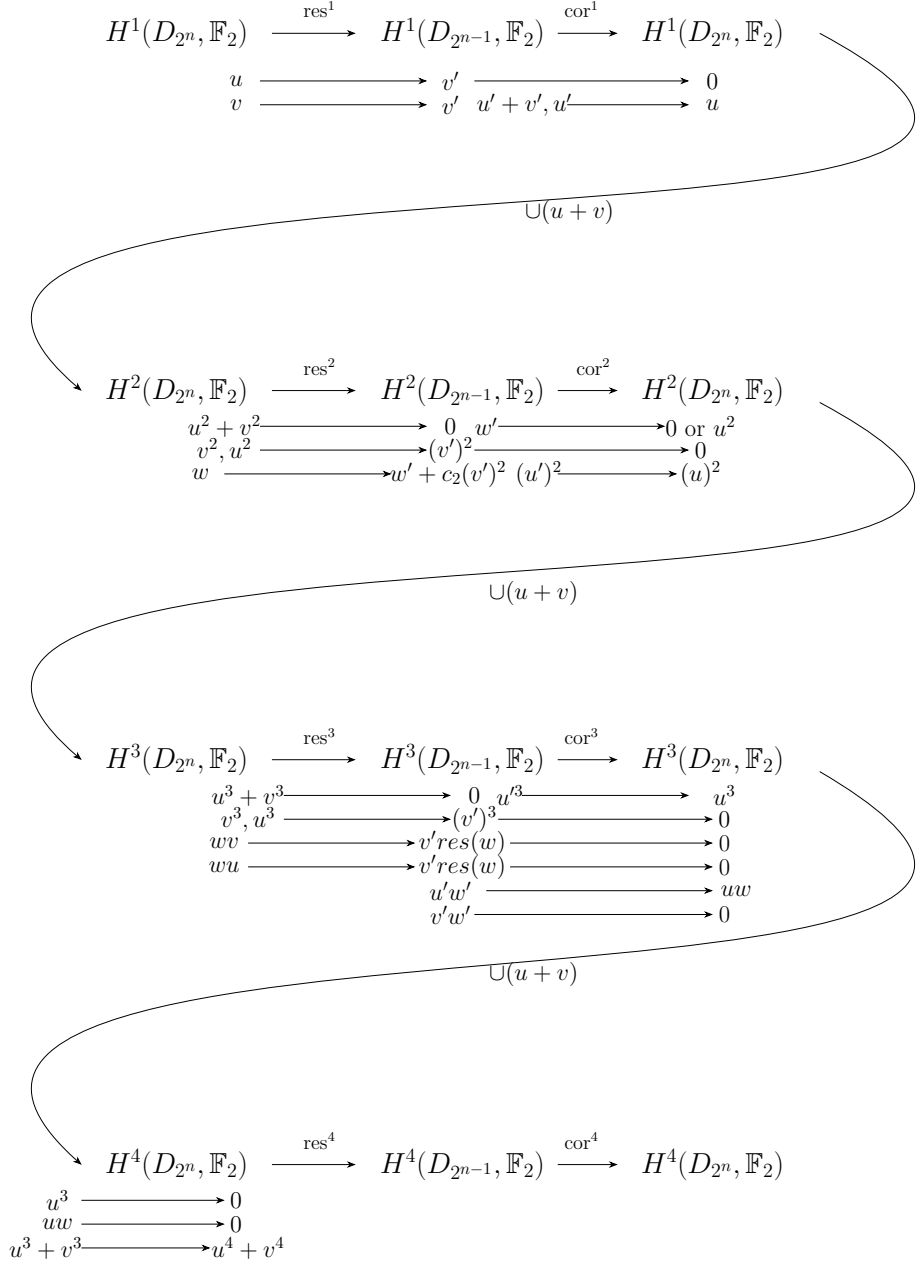
\begin{figure}[htbp]
    \centering
      \begin{center}
    \GLESfigure{%
      \begin{circuitikz}
         \tikzstyle{every node}=[font=\large]
\node [font=\Large] at (-1,16.5) {$H^1(D_{2^n}, \mathbb F_2)$};
\node [font=\Large] at (22.5,15) {$.$};
\draw [->, >=Stealth] (1,16.5) -- (2.75,16.5);
\node [font=\Large] at (5,16.5) {$H^1(D_{2^{n-1}}, \mathbb F_2)$};
\draw [->, >=Stealth] (7,16.5) -- (8.25,16.5);
\node [font=\Large] at (10.25,16.5) {$H^1(D_{2^n}, \mathbb F_2)$};
\draw [->, >=Stealth] (12.5,16.5) .. controls (22.75,11) and (-10.25,15) .. (-3,9) ;
\node [font=\Large] at (-1,9) {$H^2(D_{2^n}, \mathbb F_2)$};
\node [font=\Large] at (5,9) {$H^2(D_{2^{n-1}}, \mathbb F_2)$};
\node [font=\Large] at (10.25,9) {$H^2(D_{2^n}, \mathbb F_2)$};
\node [font=\Large] at (2.75,16) {};
\draw [->, >=Stealth] (1,9) -- (2.75,9);
\draw [->, >=Stealth] (7,9) -- (8.25,9);
\draw [->, >=Stealth] (12.5,8.75) .. controls (22.75,3.25) and (-10.25,7.25) .. (-3,1.25) ;
\node [font=\Large] at (-1,1.25) {$H^3(D_{2^n}, \mathbb F_2)$};
\node [font=\Large] at (5,1.25) {$H^3(D_{2^{n-1}}, \mathbb F_2)$};
\node [font=\Large] at (10.25,1.25) {$H^3(D_{2^n}, \mathbb F_2)$};
\draw [->, >=Stealth] (1,1.25) -- (2.75,1.25);
\draw [->, >=Stealth] (7,1.25) -- (8.25,1.25);
\draw [->, >=Stealth] (12.5,1) .. controls (22.75,-5) and (-10.25,-1) .. (-3,-7.25) ;
\node [font=\Large] at (-1,-7.25) {$H^4(D_{2^n}, \mathbb F_2)$};
\node [font=\Large] at (5,-7.25) {$H^4(D_{2^{n-1}}, \mathbb F_2)$};
\node [font=\Large] at (10.25,-7.25) {$H^4(D_{2^n}, \mathbb F_2)$};
\draw [->, >=Stealth] (1,-7.25) -- (2.75,-7.25);
\draw [->, >=Stealth] (7,-7.25) -- (8.25,-7.25);
\node [font=\normalsize] at (2,17) {$\mathrm{res}^1$};
\node [font=\normalsize] at (7.5,17) {$\mathrm{cor}^1$};
\node [font=\normalsize] at (2,9.5) {$\mathrm{res}^2$};
\node [font=\normalsize] at (7.75,9.5) {$\mathrm{cor}^2$};
\node [font=\normalsize] at (2,1.75) {$\mathrm{res}^3$};
\node [font=\normalsize] at (7.5,1.75) {$\mathrm{cor}^3$};
\node [font=\normalsize] at (2,-6.75) {$\mathrm{res}^4$};
\node [font=\normalsize] at (7.5,-6.75) {$\mathrm{cor}^4$};
\node [font=\large] at (0.25,15.5) {$u$};
\node [font=\large] at (4.75,15.5) {$v'$};

\node [font=\LARGE] at (3.25,16) {};
\node [font=\LARGE] at (3.25,16) {};
\draw [->, >=Stealth] (0.75,15.5) -- (4.25,15.5);
\node [font=\large] at (0.25,15) {$v$};
\node [font=\large] at (4.75,15) {$v'$};
\draw [->, >=Stealth] (0.75,15) -- (4.25,15);
\draw [->, >=Stealth] (5.25,15.5) -- (9.75,15.5);
\node [font=\large] at (10.25,15) {$u$};
\node [font=\large] at (10.25,15.5) {$0$};
\node [font=\large] at (6.3,15) {$u'+v', u'$};
\draw [->, >=Stealth] (7.2,15) -- (9.75,15);
\node [font=\LARGE] at (3,8.75) {};
\node [font=\large] at (0,8.25) {$u^2+v^2$};
\node [font=\large] at (4.75,8.25) {$0$};

\node [font=\LARGE] at (3,8.75) {};
\node [font=\LARGE] at (3.25,8.75) {};
\draw [->, >=Stealth] (0.75,8.25) -- (4.25,8.25);
\node [font=\large] at (0,7.75) {$v^2,u^2$};
\node [font=\large] at (4.75,7.75) {$(v')^2$};
\draw [->, >=Stealth] (0.75,7.75) -- (4.25,7.75);
\draw [->, >=Stealth] (5.25,7.75) -- (9.75,7.75);
\node [font=\large] at (10,8.25) {$0 $ or $u^2$};
\node [font=\large] at (10,7.75) {$0$};
\node [font=\large] at (5.5,8.25) {$w'$};
\draw [->, >=Stealth] (5.75,8.25) -- (9.25,8.25);

\node [font=\large] at (-0.5,7.25) {$w$};
\node [font=\large] at (4.7,7.25) {$w'+c_2(v')^2$};
\draw [->, >=Stealth] (0,7.25) -- (3.5,7.25);
\node [font=\large] at (10,7.25) {$(u)^2$};
\node [font=\large] at (6.6,7.25) {$(u')^2$};
\draw [->, >=Stealth] (7,7.25) -- (9.5,7.25);
\node [font=\large] at (0.5,0.5) {$u^3+v^3$};
\node [font=\large] at (5.25,0.5) {$0$};

\draw [->, >=Stealth] (1.25,0.5) -- (4.75,0.5);
\node [font=\large] at (0.5,0) {$v^3,u^3$};
\node [font=\large] at (5.25,0) {$(v')^3$};
\draw [->, >=Stealth] (1.25,0) -- (4.75,0);
\draw [->, >=Stealth] (5.75,0) -- (10.25,0);
\node [font=\large] at (10.5,0.5) {$u^3$};
\node [font=\large] at (10.5,0) {$0$};
\node [font=\large] at (6,0.5) {$u'^3$};
\draw [->, >=Stealth] (6.25,0.5) -- (9.75,0.5);

\node [font=\large] at (0,-0.5) {$wv$};
\node [font=\large] at (5,-0.5) {$v'res(w)$};
\draw [->, >=Stealth] (0.5,-0.5) -- (4,-0.5);
\draw [->, >=Stealth] (6,-0.5) -- (10.25,-0.5);
\node [font=\large] at (0,-1) {$wu$};
\draw [->, >=Stealth] (0.5,-1) -- (4,-1);
\node [font=\large] at (5,-1) {$v'res(w)$};
\node [font=\large] at (10.5,-0.5) {0};
\draw [->, >=Stealth] (6,-1) -- (10.25,-1);
\node [font=\large] at (10.5,-1) {0};
\node [font=\large] at (4.75,-1.5) {$u'w'$};
\node [font=\large] at (4.75,-2) {$v'w'$};
\draw [->, >=Stealth] (5.5,-1.5) -- (10.25,-1.5);
\draw [->, >=Stealth] (5.25,-2) -- (10.25,-2);
\node [font=\large] at (10.75,-1.5) {$uw$};
\node [font=\large] at (10.5,-2) {0};
\node [font=\large] at (-3.5,-8) {$u^3$};
\draw [->, >=Stealth] (-3,-8) -- (0,-8);
\node [font=\large] at (-3.5,-8.5) {$uw$};
\draw [->, >=Stealth] (-3,-8.5) -- (0,-8.5);
\node [font=\large] at (-3.75,-9) {$u^3+v^3$};
\draw [->, >=Stealth] (-3,-9) -- (-0.3,-9);
\node [font=\large] at (0.25,-8) {$0$};
\node [font=\large] at (0.25,-8.5) {$0$};
\node [font=\large] at (0.5,-9) {$u^4+v^4$};
\node [font=\large] at (7.25,12.75) {$\cup(u+v)$};
\node [font=\large] at (6.5,4.75) {$\cup(u+v)$};
\node [font=\large] at (6.5,-3.25) {$\cup(u+v)$};
      \end{circuitikz}}%
  \end{center}
    \caption{GLES $uw$ is negligible in $D_{2^n}$}
    \label{GysinD2^n(uplusv)}
\end{figure}

  \begin{figure}[htbp]       
    \centering
    \GLESfigure{%
      \begin{circuitikz}
                 \tikzstyle{every node}=[font=\large]
\node [font=\Large] at (-1.75,16.5) {$H^1(D_{2^{n+1}}, \mathbb F_2)$};

\draw [->, >=Stealth] (0.5,16.5) -- (2.5,16.5);
\node [font=\Large] at (4.75,16.5) {$H^1(\Z/2^n\Z, \mathbb F_2)$};
\draw [->, >=Stealth] (7.25,16.5) -- (8.55,16.5);
\node [font=\Large] at (10.75,16.5) {$H^1(D_{2^{n+1}}, \mathbb F_2)$};
\draw [->, >=Stealth] (13.5,16.5) .. controls (22.25,11) and (-10.75,15) .. (-4,9) ;
\node [font=\LARGE] at (3.25,16) {};
\node [font=\normalsize] at (1.75,16.9) {$\mathrm{res}^1$};
\node [font=\normalsize] at (7.75,16.8) {$\mathrm{cor}^1$};
\node [font=\large] at (0.25,15.5) {$u$};
\node [font=\large] at (4.75,15.5) {$x$};
\node [font=\LARGE] at (3.25,16) {};
\node [font=\LARGE] at (3.25,16) {};
\draw [->, >=Stealth] (0.75,15.5) -- (4.25,15.5);
\node [font=\large] at (0.25,15) {$v$};
\node [font=\large] at (4.75,15) {$0$};
\draw [->, >=Stealth] (0.75,15) -- (4.25,15);
\draw [->, >=Stealth] (5.25,15.5) -- (9.75,15.5);
\node [font=\large] at (10.25,15.5) {$0$};
\node [font=\large] at (0.25,6.5) {$w$};

\node [font=\large] at (7.25,12.75) {$\cup v$};

\node [font=\Large] at (-1.25,8.5) {$H^2(D_{2^{n+1}}, \mathbb F_2)$};
\draw [->, >=Stealth] (1,8.5) -- (3,8.5);
\node [font=\Large] at (5.25,8.5) {$H^2(\Z/2^n\Z, \mathbb F_2)$};
\draw [->, >=Stealth] (7.75,8.5) -- (9,8.5);
\node [font=\Large] at (11.25,8.5) {$H^2(D_{2^{n+1}}, \mathbb F_2)$};
\draw [->, >=Stealth] (14.25,8.5) .. controls (22.75,2.75) and (-10.25,6.75) .. (-3.75,0.5) ;
\node [font=\LARGE] at (3.75,8) {};
\node [font=\normalsize] at (2,8.75) {$\mathrm{res}^2$};
\node [font=\normalsize] at (8.25,8.75) {$\mathrm{cor}^2$};
\node [font=\large] at (0.25,7.5) {$u^2$};
\node [font=\large] at (5.5,7.5) {$x^2= 0$};

\node [font=\LARGE] at (3.75,8) {};
\node [font=\LARGE] at (3.7-5,8) {};
\draw [->, >=Stealth] (1.25,7.5) -- (4.75,7.5);
\node [font=\large] at (0.25,7) {$v^2$};
\node [font=\large] at (5.25,7) {$0$};
\draw [->, >=Stealth] (1.25,7) -- (4.75,7);
\draw [->, >=Stealth] (5.75,6.5) -- (10.25,6.5);
\node [font=\large] at (10.75,6.5) {$0$};
\node [font=\large] at (7.5,4.25) {$\cup v$};

\node [font=\large] at (-7,7.5) {$u$};
\node [font=\large] at (-2.5,7.5) {$uv$};

\draw [->, >=Stealth] (-6.5,7.5) -- (-3,7.5);
\node [font=\large] at (-7,7) {$v$};
\node [font=\large] at (-2.5,7) {$v^2$};
\draw [->, >=Stealth] (-6.5,7) -- (-3,7);

\draw [->, >=Stealth] (1.25,6.5) -- (4.75,6.5);
\node [font=\large] at (5.25,6.5) {$y$};
\node [font=\large] at (0.5,-1.75) {$wu$};
\node [font=\Large] at (-1,0.25) {$H^3(D_{2^{n+1}}, \mathbb F_2)$};
\draw [->, >=Stealth] (1.5,0.25) -- (3.25,0.25);
\node [font=\Large] at (5.5,0.25) {$H^3(\Z/2^n\Z, \mathbb F_2)$};
\draw [->, >=Stealth] (8,0.25) -- (9.25,0.25);
\node [font=\Large] at (11.5,0.25) {$H^3(D_{2^{n+1}}, \mathbb F_2)$};
\node [font=\Large] at (4,-0.25) {};
\node [font=\normalsize] at (2.5,0.5) {$\mathrm{res}^3$};
\node [font=\normalsize] at (8.5,0.5) {$\mathrm{cor}^3$};
\node [font=\large] at (0.5,-0.75) {$u^3$};
\node [font=\large] at (5.5,-0.75) {$0$};

\node [font=\LARGE] at (4,-0.25) {};
\node [font=\LARGE] at (4,-0.25) {};
\draw [->, >=Stealth] (1.5,-0.75) -- (5,-0.75);
\node [font=\large] at (0.5,-1.25) {$v^3$};
\node [font=\large] at (5.5,-1.25) {$0$};
\draw [->, >=Stealth] (1.5,-1.25) -- (5,-1.25);
\draw [->, >=Stealth] (6,-1.75) -- (10.5,-1.75);
\node [font=\large] at (11,-1.75) {$0$};

\node [font=\large] at (-2.25,-0.75) {$u^2v$};
\draw [->, >=Stealth] (-6.25,-0.75) -- (-2.75,-0.75);
\node [font=\large] at (-2.25,-1.25) {$v^3$};
\draw [->, >=Stealth] (-6.25,-1.25) -- (-2.75,-1.25);
\draw [->, >=Stealth] (1.5,-1.75) -- (5,-1.75);
\node [font=\large] at (5.5,-1.75) {$xy$};
\node [font=\large] at (-6.5,-0.75) {$u^2$};
\node [font=\large] at (-6.5,-1.75) {$w$};
\node [font=\large] at (-6.5,-1.25) {$v^2$};
\draw [->, >=Stealth] (-6.25,-1.75) -- (-2.75,-1.75);
\node [font=\large] at (-2.25,-1.75) {$vw$};
\node [font=\large] at (0.5,-2.25) {$wv$};
\draw [->, >=Stealth] (1.5,-2.25) -- (5,-2.25);
\node [font=\large] at (5.5,-2.25) {$0$};
      \end{circuitikz}
    }
    \caption{Gysin long exact sequence showing $xy$ is negligible in $\mathbb{Z}/2^n \mathbb{Z}$}
    \label{Gysincyclic}
  \end{figure}

      \begin{figure}[htbp]  
    \centering
    \GLESfigure{%
      \begin{circuitikz}
                 \tikzstyle{every node}=[font=\large]
\node [font=\Large] at (-1.75,16.5) {$H^1(Q_{16}, \mathbb F_2)$};

\draw [->, >=Stealth] (0.5,16.5) -- (2.5,16.5);
\node [font=\Large] at (4.75,16.5) {$H^1(C_{8 }, \mathbb F_2)$};
\draw [->, >=Stealth] (7.25,16.5) -- (8.55,16.5);
\node [font=\Large] at (10.75,16.5) {$H^1(Q_{16}, \mathbb F_2)$};
\draw [->, >=Stealth] (13.5,16.5) .. controls (22.25,11) and (-10.75,15) .. (-4,9) ;
\node [font=\LARGE] at (3.25,16) {};
\node [font=\normalsize] at (1.75,16.9) {$\mathrm{res}^1$};
\node [font=\normalsize] at (7.75,16.9) {$\mathrm{cor}^1$};
\node [font=\large] at (0.25,15.5) {$a+b$};
\node [font=\large] at (4.5,15.5) {$0$};

\node [font=\LARGE] at (3.25,16) {};
\node [font=\LARGE] at (3.25,16) {};
\draw [->, >=Stealth] (0.75,15.5) -- (4.25,15.5);

\node [font=\large] at (-0.75,15) {$a,b$};
\node [font=\large] at (4.75,15) {$x$};
\draw [->, >=Stealth] (0.75,15) -- (4.25,15);

\draw [->, >=Stealth] (5.25,15) -- (9.75,15);
\node [font=\large] at (10.25,15) {$0$};

\node [font=\large] at (0.25,6.5) {$a^2+b^2$};
\node [font=\large] at (0.25,7.5) {$a^2,b^2$};
\draw [->, >=Stealth] (0.75,7.5) -- (4.75,7.5);
\node [font=\large] at (5.25,7.5) {$0$};
\node [font=\large] at (7.25,13.5) {$\cup (a+b )$};

\node [font=\Large] at (-1.25,8.5) {$H^2(Q_{16}, \mathbb F_2)$};
\draw [->, >=Stealth] (1,8.5) -- (3,8.5);
\node [font=\Large] at (5.25,8.5) {$H^2(C_{8}, \mathbb F_2)$};
\draw [->, >=Stealth] (7.75,8.5) -- (9,8.5);
\node [font=\Large] at (11.25,8.5) {$H^2(Q_{16}, \mathbb F_2)$};
\draw [->, >=Stealth] (14.25,8.5) .. controls (22.75,2.75) and (-10.25,6.75) .. (-3.75,0.5) ;
\node [font=\LARGE] at (3.75,8) {};
\node [font=\normalsize] at (2,8.75) {$\mathrm{res}^2$};
\node [font=\normalsize] at (8.25,8.75) {$\mathrm{cor}^2$};

\node [font=\LARGE] at (3.75,8) {};
\node [font=\LARGE] at (3.7-5,8) {};

\node [font=\large] at (7.5,5.25) {$\cup (a+b)$};
\node [font=\large] at (5.75,7.5) {$y$};
\draw [->, >=Stealth] (6.25,7.5) -- (11.75,7.5);

\node [font=\large] at (12.75,7.5) {$a^2+b^2$};

\draw [->, >=Stealth] (1.25,6.5) -- (4.75,6.5);
\node [font=\large] at (5.25,6.5) {$0$};

\node [font=\Large] at (-1,0.25) {$H^3(Q_{16}, \mathbb F_2)$};
\draw [->, >=Stealth] (1.5,0.25) -- (3.25,0.25);
\node [font=\Large] at (5.5,0.25) {$H^3(C_{8}, \mathbb F_2)$};
\draw [->, >=Stealth] (8,0.25) -- (9.25,0.25);
\node [font=\Large] at (11.5,0.25) {$H^3(Q_{16}, \mathbb F_2)$};
\node [font=\Large] at (4,-0.25) {};
\node [font=\normalsize] at (2.5,0.5) {$\mathrm{res}^3$};
\node [font=\normalsize] at (8.5,0.5) {$\mathrm{cor}^3$};

\node [font=\large] at (5.5,-1.75) {$0$};

\node [font=\LARGE] at (4,-0.25) {};
\node [font=\LARGE] at (4,-0.25) {};

\draw [->, >=Stealth] (-1,-1.75) -- (5,-1.75);

\node [font=\large] at (6.75,-1.75) {$xy$};
\node [font=\large] at (11.5,-1.75) {$\text{cor}(\text{res}(a)y)=a^3=b^3$};
\draw [->, >=Stealth] (7.25,-1.75) -- (9,-1.75);

\node [font=\large] at (-2,-1.75) {$a^3$};

      \end{circuitikz}
    }
    \caption{GLES $a^3=b^3$ is negligible in $Q_{16}$}
    \label{Q16resC8}
  \end{figure}

      \begin{figure}[htbp] 
    \centering
    \GLESfigure{%
      \begin{circuitikz}
                 \tikzstyle{every node}=[font=\large]
\node [font=\Large] at (-1.75,16.5) {$H^1(Q_{16}, \mathbb F_2)$};

\draw [->, >=Stealth] (0.5,16.5) -- (2.5,16.5);
\node [font=\Large] at (4.75,16.5) {$H^1(Q_{8}, \mathbb F_2)$};
\draw [->, >=Stealth] (7.25,16.5) -- (8.55,16.5);
\node [font=\Large] at (10.75,16.5) {$H^1(Q_{16}, \mathbb F_2)$};
\draw [->, >=Stealth] (13.5,16.5) .. controls (22.25,11) and (-10.75,15) .. (-4,9) ;
\node [font=\LARGE] at (3.25,16) {};
\node [font=\normalsize] at (1.75,16.9) {$\mathrm{res}^1$};
\node [font=\normalsize] at (7.75,16.9) {$\mathrm{cor}^1$};
\node [font=\large] at (0.25,15.5) {$a$};
\node [font=\large] at (4.5,15.5) {$0$};
\node [font=\large] at (5.25,15.5) {$\ast,\ast$};
\node [font=\LARGE] at (3.25,16) {};
\node [font=\LARGE] at (3.25,16) {};
\draw [->, >=Stealth] (0.75,15.5) -- (4.25,15.5);
\draw [->, >=Stealth] (6,15.5) -- (9.75,15.5);
\node [font=\large] at (10,15.5) {$b$};
\node [font=\large] at (-0.75,15) {$a+b, b$};
\node [font=\large] at (4.75,15) {$\ast$};
\draw [->, >=Stealth] (0.75,15) -- (4.25,15);

\draw [->, >=Stealth] (5.25,15) -- (9.75,15);
\node [font=\large] at (10.25,15) {$0$};

\node [font=\large] at (0.25,6.5) {$\ast,\ast$};

\node [font=\large] at (7.25,13.5) {$\cup (b)$};

\node [font=\Large] at (-1.25,8.5) {$H^2(Q_{16}, \mathbb F_2)$};
\draw [->, >=Stealth] (1,8.5) -- (3,8.5);
\node [font=\Large] at (5.25,8.5) {$H^2(Q_{8}, \mathbb F_2)$};
\draw [->, >=Stealth] (7.75,8.5) -- (9,8.5);
\node [font=\Large] at (11.25,8.5) {$H^2(Q_{16}, \mathbb F_2)$};
\draw [->, >=Stealth] (14.25,8.5) .. controls (22.75,2.75) and (-10.25,6.75) .. (-3.75,0.5) ;
\node [font=\LARGE] at (3.75,8) {};
\node [font=\normalsize] at (2,8.75) {$\mathrm{res}^2$};
\node [font=\normalsize] at (8.25,8.75) {$\mathrm{cor}^2$};

\node [font=\LARGE] at (3.75,8) {};
\node [font=\LARGE] at (3.7-5,8) {};

\node [font=\large] at (7.5,5.25) {$\cup (b)$};
\node [font=\large] at (5.75,7.5) {$y$};
\draw [->, >=Stealth] (6.25,7.5) -- (11.75,7.5);

\node [font=\large] at (12.75,7.5) {$a^2+b^2$};

\draw [->, >=Stealth] (1.25,6.5) -- (4.75,6.5);
\node [font=\large] at (5.25,6.5) {$0$};

\node [font=\Large] at (-1,0.25) {$H^3(Q_{16}, \mathbb F_2)$};
\draw [->, >=Stealth] (1.5,0.25) -- (3.25,0.25);
\node [font=\Large] at (5.5,0.25) {$H^3(Q_{8}, \mathbb F_2)$};
\draw [->, >=Stealth] (8,0.25) -- (9.25,0.25);
\node [font=\Large] at (11.5,0.25) {$H^3(Q_{16}, \mathbb F_2)$};
\node [font=\Large] at (4,-0.25) {};
\node [font=\normalsize] at (2.5,0.5) {$\mathrm{res}^3$};
\node [font=\normalsize] at (8.5,0.5) {$\mathrm{cor}^3$};

\node [font=\large] at (5.5,-1.75) {$0$};

\node [font=\LARGE] at (4,-0.25) {};
\node [font=\LARGE] at (4,-0.25) {};

\draw [->, >=Stealth] (-1,-1.75) -- (5,-1.75);

\node [font=\large] at (-2,-1.75) {$a^3$};

\node [font=\Large] at (-1, -8) {$H^4(Q_{16}, \mathbb F_2)$};
\draw [->, >=Stealth] (1.5,-8) -- (3.5,-8);
\node [font=\Large] at (6,-8) {$H^4(Q_{8}, \mathbb F_2)$};
\draw [->, >=Stealth] (8.5,-8) -- (10,-8);
\node [font=\Large] at (12.5,-8) {$H^4(Q_{16}, \mathbb F_2)$};
\node [font=\normalsize] at (2.5,-7.5) {$\mathrm{res}^4$};
\node [font=\normalsize] at (9,-7.5) {$\mathrm{cor}^4$};

\node [font=\large] at (-0.5,-9.25) {$c$};
\node [font=\large] at (6,-9.25) {$z$};
\draw [->, >=Stealth] (-0.25,-9.25) -- (5.5,-9.25);
\draw [->, >=Stealth] (6.25,-9.25) -- (11,-9.25);
\node [font=\large] at (11.5,-9.25) {$0$};

\draw [->, >=Stealth] (14.5,-8) .. controls (20,-12.5) and (-10,-13) .. (-4,-16);
\node [font=\large] at (7.5,-11.5) {$\cup (b)$};

\node [font=\Large] at (-1, -16.5) {$H^5(Q_{16}, \mathbb F_2)$};
\draw [->, >=Stealth] (1.5,-16.5) -- (3.5,-16.5);
\node [font=\Large] at (6,-16.5) {$H^5(Q_{8}, \mathbb F_2)$};
\draw [->, >=Stealth] (8.5,-16.5) -- (10,-16.5);
\node [font=\Large] at (12.5,-16.5) {$H^5(Q_{16}, \mathbb F_2)$};
\node [font=\normalsize] at (2.5,-16) {$\mathrm{res}^5$};
\node [font=\normalsize] at (9,-16) {$\mathrm{cor}^5$};

\node [font=\large] at (6.5,-17.75) {$\ast z$};

\draw [->, >=Stealth] (7.25,-17.75) -- (11.75,-17.75);
\node [font=\large] at (13.5,-17.75) {$ \text{cor}(\ast \text{res}(c))=bc$};

\draw [->, >=Stealth] (13.5,0) .. controls (20,-4.5) and (-10,-3.5) .. (-4,-7.5);
\node [font=\large] at (7.5,-3.25) {$\cup (b)$};

      \end{circuitikz}
    }
    \caption{GLES $bc$ is negligible in $Q_{16}$}
    \label{Q16resQ8-1}
  \end{figure}
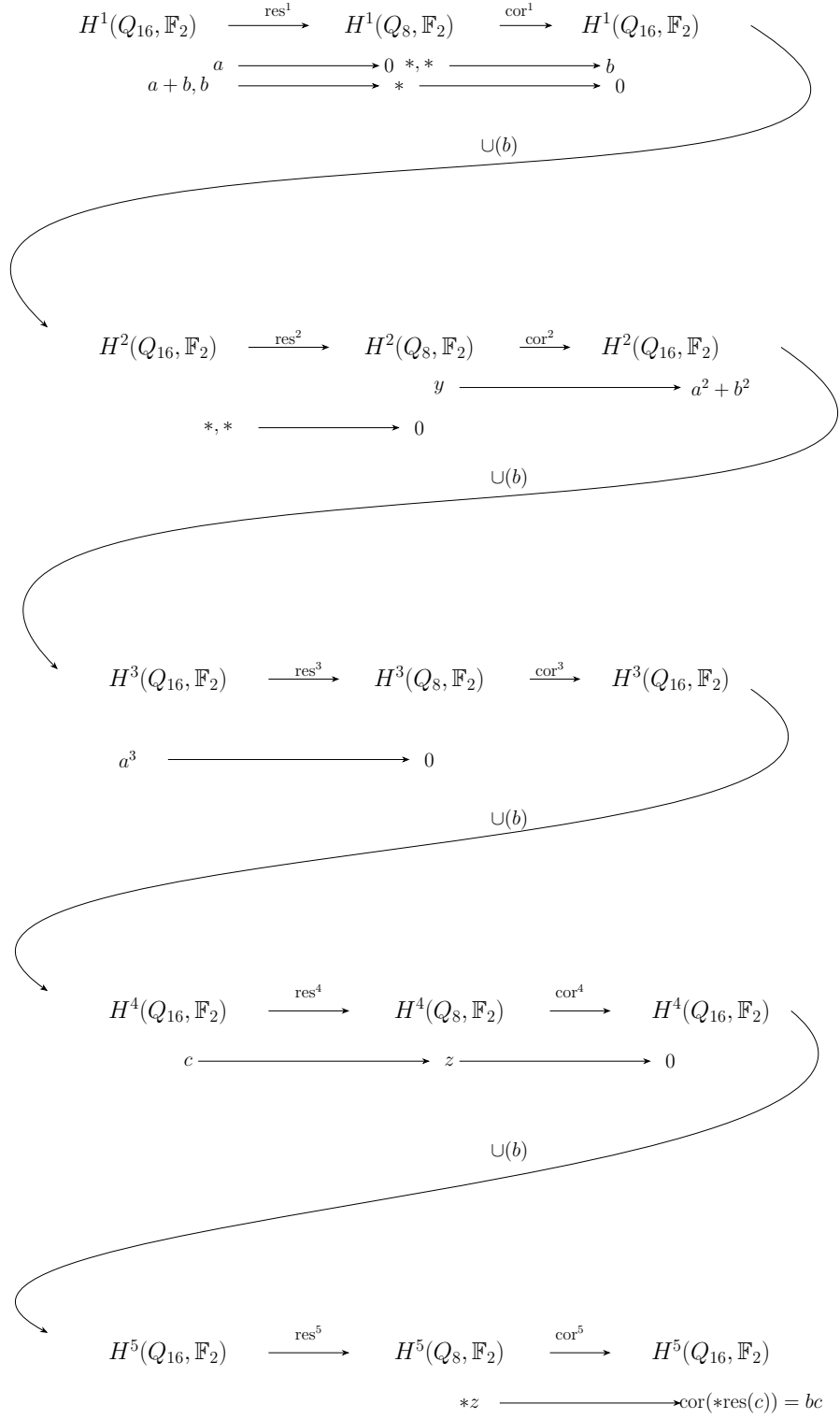

      \begin{figure}[htbp]       
    \centering
    \resizebox{\textwidth}{!}{%
      \begin{circuitikz}
                 \tikzstyle{every node}=[font=\large]
\node [font=\Large] at (-1.75,16.5) {$H^1(Q_{16}, \mathbb F_2)$};

\draw [->, >=Stealth] (0.5,16.5) -- (2.5,16.5);
\node [font=\Large] at (4.75,16.5) {$H^1(Q_{8}, \mathbb F_2)$};
\draw [->, >=Stealth] (7.25,16.5) -- (8.55,16.5);
\node [font=\Large] at (10.75,16.5) {$H^1(Q_{16}, \mathbb F_2)$};
\draw [->, >=Stealth] (13.5,16.5) .. controls (22.25,11) and (-10.75,15) .. (-4,9) ;
\node [font=\LARGE] at (3.25,16) {};
\node [font=\normalsize] at (1.75,16.9) {$\mathrm{res}^1$};
\node [font=\normalsize] at (7.75,16.9) {$\mathrm{cor}^1$};
\node [font=\large] at (0.25,15.5) {$b$};
\node [font=\large] at (4.5,15.5) {$0$};
\node [font=\large] at (5.25,15.5) {$\ast,\ast$};
\node [font=\LARGE] at (3.25,16) {};
\node [font=\LARGE] at (3.25,16) {};
\draw [->, >=Stealth] (0.75,15.5) -- (4.25,15.5);
\draw [->, >=Stealth] (6,15.5) -- (9.75,15.5);
\node [font=\large] at (10,15.5) {$a$};
\node [font=\large] at (-0.75,15) {$a+b, a$};
\node [font=\large] at (4.75,15) {$\ast$};
\draw [->, >=Stealth] (0.75,15) -- (4.25,15);

\draw [->, >=Stealth] (5.25,15) -- (9.75,15);
\node [font=\large] at (10.25,15) {$0$};

\node [font=\large] at (0.25,6.5) {$\ast,\ast$};

\node [font=\large] at (7.25,13.5) {$\cup (a)$};

\node [font=\Large] at (-1.25,8.5) {$H^2(Q_{16}, \mathbb F_2)$};
\draw [->, >=Stealth] (1,8.5) -- (3,8.5);
\node [font=\Large] at (5.25,8.5) {$H^2(Q_{8}, \mathbb F_2)$};
\draw [->, >=Stealth] (7.75,8.5) -- (9,8.5);
\node [font=\Large] at (11.25,8.5) {$H^2(Q_{16}, \mathbb F_2)$};
\draw [->, >=Stealth] (14.25,8.5) .. controls (22.75,2.75) and (-10.25,6.75) .. (-3.75,0.5) ;
\node [font=\LARGE] at (3.75,8) {};
\node [font=\normalsize] at (2,8.75) {$\mathrm{res}^2$};
\node [font=\normalsize] at (8.25,8.75) {$\mathrm{cor}^2$};

\node [font=\LARGE] at (3.75,8) {};
\node [font=\LARGE] at (3.7-5,8) {};

\node [font=\large] at (7.5,5.25) {$\cup (a)$};
\node [font=\large] at (5.75,7.5) {$y$};
\draw [->, >=Stealth] (6.25,7.5) -- (11.75,7.5);

\node [font=\large] at (12.75,7.5) {$a^2+b^2$};

\draw [->, >=Stealth] (1.25,6.5) -- (4.75,6.5);
\node [font=\large] at (5.25,6.5) {$0$};

\node [font=\Large] at (-1,0.25) {$H^3(Q_{16}, \mathbb F_2)$};
\draw [->, >=Stealth] (1.5,0.25) -- (3.25,0.25);
\node [font=\Large] at (5.5,0.25) {$H^3(Q_{8}, \mathbb F_2)$};
\draw [->, >=Stealth] (8,0.25) -- (9.25,0.25);
\node [font=\Large] at (11.5,0.25) {$H^3(Q_{16}, \mathbb F_2)$};
\node [font=\Large] at (4,-0.25) {};
\node [font=\normalsize] at (2.5,0.5) {$\mathrm{res}^3$};
\node [font=\normalsize] at (8.5,0.5) {$\mathrm{cor}^3$};

\node [font=\large] at (5.5,-1.75) {$0$};

\node [font=\LARGE] at (4,-0.25) {};
\node [font=\LARGE] at (4,-0.25) {};

\draw [->, >=Stealth] (-1,-1.75) -- (5,-1.75);

\node [font=\large] at (-2,-1.75) {$b^3$};

\node [font=\Large] at (-1, -8) {$H^4(Q_{16}, \mathbb F_2)$};
\draw [->, >=Stealth] (1.5,-8) -- (3.5,-8);
\node [font=\Large] at (6,-8) {$H^4(Q_{8}, \mathbb F_2)$};
\draw [->, >=Stealth] (8.5,-8) -- (10,-8);
\node [font=\Large] at (12.5,-8) {$H^4(Q_{16}, \mathbb F_2)$};
\node [font=\normalsize] at (2.5,-7.5) {$\mathrm{res}^4$};
\node [font=\normalsize] at (9,-7.5) {$\mathrm{cor}^4$};

\node [font=\large] at (-0.5,-9.25) {$c$};
\node [font=\large] at (6,-9.25) {$z$};
\draw [->, >=Stealth] (-0.25,-9.25) -- (5.5,-9.25);
\draw [->, >=Stealth] (6.25,-9.25) -- (11,-9.25);
\node [font=\large] at (11.5,-9.25) {$0$};

\draw [->, >=Stealth] (14.5,-8) .. controls (20,-12.5) and (-10,-13) .. (-4,-16);
\node [font=\large] at (7.5,-11.5) {$\cup (a)$};

\node [font=\Large] at (-1, -16.5) {$H^5(Q_{16}, \mathbb F_2)$};
\draw [->, >=Stealth] (1.5,-16.5) -- (3.5,-16.5);
\node [font=\Large] at (6,-16.5) {$H^5(Q_{8}, \mathbb F_2)$};
\draw [->, >=Stealth] (8.5,-16.5) -- (10,-16.5);
\node [font=\Large] at (12.5,-16.5) {$H^5(Q_{16}, \mathbb F_2)$};
\node [font=\normalsize] at (2.5,-16) {$\mathrm{res}^5$};
\node [font=\normalsize] at (9,-16) {$\mathrm{cor}^5$};

\node [font=\large] at (6.5,-17.75) {$\ast z$};

\draw [->, >=Stealth] (7.25,-17.75) -- (11.75,-17.75);
\node [font=\large] at (13.5,-17.75) {$\text{cor}(\ast\text{res}(c))=ac$};

\draw [->, >=Stealth] (13.5,0) .. controls (20,-4.5) and (-10,-3.5) .. (-4,-7.5);
\node [font=\large] at (7.5,-3.25) {$\cup (a)$};

      \end{circuitikz}
    }
    \caption{GLES $ac$ is negligible in $Q_{16}$}
    \label{Q16resQ8-2}
  \end{figure}
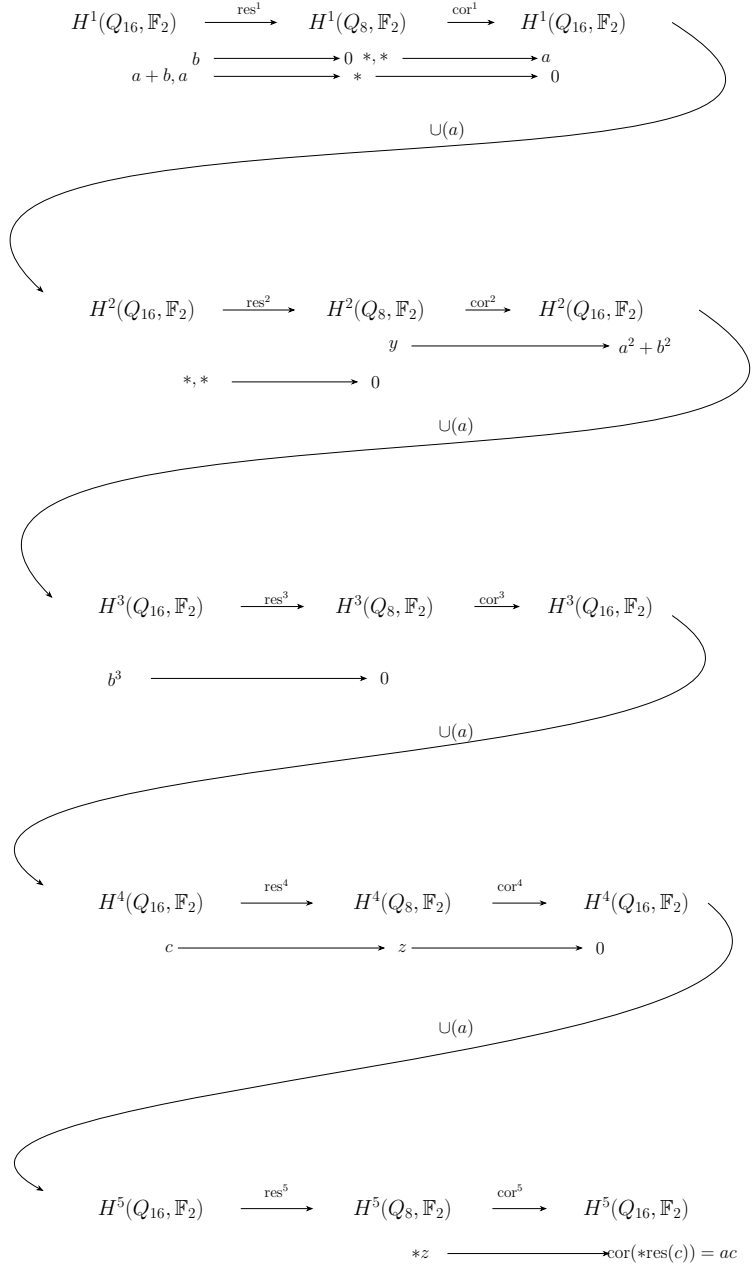

      \begin{figure}[htbp]      
    \centering
    \GLESfigure{%
      \begin{circuitikz}
                 \tikzstyle{every node}=[font=\large]
\node [font=\Large] at (-1.75,16.5) {$H^1(Q_{32}, \mathbb F_2)$};

\draw [->, >=Stealth] (0.5,16.5) -- (2.5,16.5);
\node [font=\Large] at (4.75,16.5) {$H^1(C_{16}, \mathbb F_2)$};
\draw [->, >=Stealth] (7.25,16.5) -- (8.55,16.5);
\node [font=\Large] at (10.75,16.5) {$H^1(Q_{32}, \mathbb F_2)$};
\draw [->, >=Stealth] (13.5,16.5) .. controls (22.25,11) and (-10.75,15) .. (-4,9) ;
\node [font=\LARGE] at (3.25,16) {};
\node [font=\normalsize] at (1.75,16.9) {$\mathrm{res}^1$};
\node [font=\normalsize] at (7.75,16.8) {$\mathrm{cor}^1$};
\node [font=\large] at (0.25,15.5) {$b$};
\node [font=\large] at (5.25,15.5) {$x$};
\node [font=\LARGE] at (3.25,16) {};
\node [font=\LARGE] at (3.25,16) {};
\draw [->, >=Stealth] (0.75,15.5) -- (4.25,15.5);
\draw [->, >=Stealth] (6,15.5) -- (9.75,15.5);
\node [font=\large] at (10,15.5) {$0$};
\node [font=\large] at (0.25,15) {$a$};
\node [font=\large] at (4.75,15) {$x$};
\draw [->, >=Stealth] (0.75,15) -- (4.25,15);

\node [font=\large] at (0.25,6.5) {$a^2+b^2$};

\node [font=\large] at (7.25,13.5) {$\cup (a+b)$};

\node [font=\Large] at (-1.25,8.5) {$H^2(Q_{32}, \mathbb F_2)$};
\draw [->, >=Stealth] (1,8.5) -- (3,8.5);
\node [font=\Large] at (5.25,8.5) {$H^2(C_{16}, \mathbb F_2)$};
\draw [->, >=Stealth] (7.75,8.5) -- (9,8.5);
\node [font=\Large] at (11.25,8.5) {$H^2(Q_{32}, \mathbb F_2)$};
\draw [->, >=Stealth] (14.25,8.5) .. controls (22.75,2.75) and (-10.25,6.75) .. (-3.75,0.5) ;
\node [font=\LARGE] at (3.75,8) {};
\node [font=\normalsize] at (2,8.75) {$\mathrm{res}^2$};
\node [font=\normalsize] at (8.25,8.75) {$\mathrm{cor}^2$};
\node [font=\large] at (0.25,7.5) {$a^2,b^2$};
\node [font=\large] at (5.75,7.5) {$0$};

\node [font=\LARGE] at (3.75,8) {};
\node [font=\LARGE] at (3.7-5,8) {};
\draw [->, >=Stealth] (1.25,7.5) -- (4.75,7.5);

\node [font=\large] at (12.75,7.5) {$a^2+b^2$};
\node [font=\large] at (7.5,5.25) {$\cup (a+b)$};
\node [font=\large] at (7.75,7.5) {$y$};
\draw [->, >=Stealth] (8.25,7.5) -- (11.75,7.5);

\node [font=\large] at (-7,7.5) {$a$};

\node [font=\large] at (-2.1,7.5) {$a^2+b^2$};

\draw [->, >=Stealth] (-6.5,7.5) -- (-3,7.5);
\node [font=\large] at (-7,7) {$b$};
\node [font=\large] at (-2.1,7) {$b^2$};
\draw [->, >=Stealth] (-6.5,7) -- (-3,7);

\draw [->, >=Stealth] (1.25,6.5) -- (4.75,6.5);
\node [font=\large] at (5.25,6.5) {$0$};

\node [font=\Large] at (-1,0.25) {$H^3(Q_{32}, \mathbb F_2)$};
\draw [->, >=Stealth] (1.5,0.25) -- (3.25,0.25);
\node [font=\Large] at (5.5,0.25) {$H^3(C_{16}, \mathbb F_2)$};
\draw [->, >=Stealth] (8,0.25) -- (9.25,0.25);
\node [font=\Large] at (11.5,0.25) {$H^3(Q_{32}, \mathbb F_2)$};
\node [font=\Large] at (4,-0.25) {};
\node [font=\normalsize] at (2.5,0.5) {$\mathrm{res}^3$};
\node [font=\normalsize] at (8.5,0.5) {$\mathrm{cor}^3$};

\node [font=\large] at (5.5,-0.75) {$0$};

\node [font=\LARGE] at (4,-0.25) {};
\node [font=\LARGE] at (4,-0.25) {};
\draw [->, >=Stealth] (-1,-0.75) -- (5,-0.75);

\node [font=\large] at (-1.75,-0.75) {$a^3=b^3$};
\draw [->, >=Stealth] (-6.25,-0.75) -- (-2.75,-0.75);

\node [font=\large] at (-6.5,-0.75) {$a^2$};

\node [font=\large] at (-6.5,-1.75) {$b^2$};

\node [font=\large] at (6.75,-1.75) {$xy$};
\node [font=\large] at (11.25,-1.75) {$a^3=b^3$};
\draw [->, >=Stealth] (7.5,-1.75) -- (10.5,-1.75);
\draw [->, >=Stealth] (-6.25,-1.75) -- (-2.75,-1.75);
\node [font=\large] at (-2,-1.75) {$0$};

      \end{circuitikz}
    }
    \caption{GLES $a^3=b^3$ is negligible in $Q_{32}$}
    \label{Q32resC16}
  \end{figure}

      \begin{figure}[htbp]     
    \centering
    \GLESfigure{%
      \begin{circuitikz}
                 \tikzstyle{every node}=[font=\large]
\node [font=\Large] at (-1.75,16.5) {$H^1(Q_{32}, \mathbb F_2)$};

\draw [->, >=Stealth] (0.5,16.5) -- (2.5,16.5);
\node [font=\Large] at (4.75,16.5) {$H^1(Q_{16}, \mathbb F_2)$};
\draw [->, >=Stealth] (7.25,16.5) -- (8.55,16.5);
\node [font=\Large] at (10.75,16.5) {$H^1(Q_{32}, \mathbb F_2)$};
\draw [->, >=Stealth] (13.5,16.5) .. controls (22.25,11) and (-10.75,15) .. (-4,9) ;
\node [font=\LARGE] at (3.25,16) {};
\node [font=\normalsize] at (1.75,16.9) {$\mathrm{res}^1$};
\node [font=\normalsize] at (7.75,16.9) {$\mathrm{cor}^1$};
\node [font=\large] at (0.25,15.5) {$a$};
\node [font=\large] at (5.25,15.5) {$a'+b'$};
\node [font=\LARGE] at (3.25,16) {};
\node [font=\LARGE] at (3.25,16) {};
\draw [->, >=Stealth] (0.75,15.5) -- (4.25,15.5);
\draw [->, >=Stealth] (6,15.5) -- (9.75,15.5);
\node [font=\large] at (10,15.5) {$0$};
\node [font=\large] at (0.25,15) {$b$};
\node [font=\large] at (4.75,15) {$0$};
\draw [->, >=Stealth] (0.75,15) -- (4.25,15);
\node [font=\large] at (5.75,15) {$a'$};
\node [font=\large] at (5.75,14.5) {$b'$};
\draw [->, >=Stealth] (6.25,14.5) -- (9.75,14.5);
\draw [->, >=Stealth] (6.25,15) -- (9.75,15);
\node [font=\large] at (10.25,15) {$a$};
\node [font=\large] at (10.25,14.5) {$a$};
\node [font=\large] at (0.25,6.5) {$b^2$};

\node [font=\large] at (7.25,13.5) {$\cup (b)$};

\node [font=\Large] at (-1.25,8.5) {$H^2(Q_{32}, \mathbb F_2)$};
\draw [->, >=Stealth] (1,8.5) -- (3,8.5);
\node [font=\Large] at (5.25,8.5) {$H^2(Q_{16}, \mathbb F_2)$};
\draw [->, >=Stealth] (7.75,8.5) -- (9,8.5);
\node [font=\Large] at (11.25,8.5) {$H^2(Q_{32}, \mathbb F_2)$};
\draw [->, >=Stealth] (14.25,8.5) .. controls (22.75,2.75) and (-10.25,6.75) .. (-3.75,0.5) ;
\node [font=\LARGE] at (3.75,8) {};
\node [font=\normalsize] at (2,8.75) {$\mathrm{res}^2$};
\node [font=\normalsize] at (8.25,8.75) {$\mathrm{cor}^2$};
\node [font=\large] at (0.25,7.5) {$a^2$};
\node [font=\large] at (5.75,7.5) {$a'^2+b'^2$};

\node [font=\LARGE] at (3.75,8) {};
\node [font=\LARGE] at (3.7-5,8) {};
\draw [->, >=Stealth] (1.25,7.5) -- (4.75,7.5);

\node [font=\large] at (12.25,7.5) {$a^2$};
\node [font=\large] at (7.5,5.25) {$\cup (b)$};
\node [font=\large] at (7.75,7.5) {$a',b'$};
\draw [->, >=Stealth] (8.25,7.5) -- (11.75,7.5);

\node [font=\large] at (-7,7.5) {$a$};
\node [font=\large] at (-7,6.5) {$a+b$};
\node [font=\large] at (-2.2,6.5) {$b^2$};
\draw [->, >=Stealth] (-6.3,6.5) -- (-3,6.5);
\node [font=\large] at (-2.1,7.5) {$0$};

\draw [->, >=Stealth] (-6.5,7.5) -- (-3,7.5);
\node [font=\large] at (-7,7) {$b$};
\node [font=\large] at (-2.1,7) {$b^2$};
\draw [->, >=Stealth] (-6.5,7) -- (-3,7);

\draw [->, >=Stealth] (1.25,6.5) -- (4.75,6.5);
\node [font=\large] at (5.25,6.5) {$0$};

\node [font=\Large] at (-1,0.25) {$H^3(Q_{32}, \mathbb F_2)$};
\draw [->, >=Stealth] (1.5,0.25) -- (3.25,0.25);
\node [font=\Large] at (5.5,0.25) {$H^3(Q_{16}, \mathbb F_2)$};
\draw [->, >=Stealth] (8,0.25) -- (9.25,0.25);
\node [font=\Large] at (11.5,0.25) {$H^3(Q_{32}, \mathbb F_2)$};
\node [font=\Large] at (4,-0.25) {};
\node [font=\normalsize] at (2.5,0.5) {$\mathrm{res}^3$};
\node [font=\normalsize] at (8.5,0.5) {$\mathrm{cor}^3$};

\node [font=\large] at (5.5,-0.75) {$0$};
\node [font=\large] at (5.5,-1.75) {$0$};

\node [font=\LARGE] at (4,-0.25) {};
\node [font=\LARGE] at (4,-0.25) {};
\draw [->, >=Stealth] (-1,-0.75) -- (5,-0.75);

\node [font=\large] at (-2,-0.75) {$0$};
\draw [->, >=Stealth] (-6.25,-0.75) -- (-2.75,-0.75);

\draw [->, >=Stealth] (-1,-1.75) -- (5,-1.75);
 \draw [->, >=Stealth] (-6.5,-2.75) -- (-2.75,-2.75);

\node [font=\large] at (-6.5,-0.75) {$a^2$};
\node [font=\large] at (-7.75,-2.75) {$a^2+b^2$};
\node [font=\large] at (-2,-2.75) {$b^3$};
\node [font=\large] at (0,-2.75) {$b^3=a^3$};
\draw [->, >=Stealth] (1.5,-2.75) -- (4.5,-2.75);
\node [font=\large] at (-6.5,-1.75) {$b^2$};
\node [font=\large] at (6.25,-2.75) {$a'^3+b'^3=0$};
\node [font=\large] at (6.75,-1.75) {$a'^3=b'^3$};
\node [font=\large] at (10.75,-1.75) {$a^3$};
\draw [->, >=Stealth] (8,-1.75) -- (10,-1.75);
\draw [->, >=Stealth] (-6.25,-1.75) -- (-2.75,-1.75);
\node [font=\large] at (-2,-1.75) {$b^3$};

\node [font=\Large] at (-1, -8) {$H^4(Q_{32}, \mathbb F_2)$};
\draw [->, >=Stealth] (1.5,-8) -- (3.5,-8);
\node [font=\Large] at (6,-8) {$H^4(Q_{16}, \mathbb F_2)$};
\draw [->, >=Stealth] (8.5,-8) -- (10,-8);
\node [font=\Large] at (12.5,-8) {$H^4(Q_{32}, \mathbb F_2)$};
\node [font=\normalsize] at (2.5,-7.5) {$\mathrm{res}^4$};
\node [font=\normalsize] at (9,-7.5) {$\mathrm{cor}^4$};

\node [font=\large] at (-0.5,-9.25) {$c$};
\node [font=\large] at (6,-9.25) {$c'$};
\draw [->, >=Stealth] (-0.25,-9.25) -- (5.5,-9.25);
\draw [->, >=Stealth] (6.25,-9.25) -- (11,-9.25);
\node [font=\large] at (11.5,-9.25) {$0$};

\draw [->, >=Stealth] (14.5,-8) .. controls (20,-12.5) and (-10,-13) .. (-4,-16);
\node [font=\large] at (7.5,-11.5) {$\cup (b)$};

\node [font=\Large] at (-1, -16.5) {$H^5(Q_{32}, \mathbb F_2)$};
\draw [->, >=Stealth] (1.5,-16.5) -- (3.5,-16.5);
\node [font=\Large] at (6,-16.5) {$H^5(Q_{16}, \mathbb F_2)$};
\draw [->, >=Stealth] (8.5,-16.5) -- (10,-16.5);
\node [font=\Large] at (12.5,-16.5) {$H^5(Q_{32}, \mathbb F_2)$};
\node [font=\normalsize] at (2.5,-16) {$\mathrm{res}^5$};
\node [font=\normalsize] at (9,-16) {$\mathrm{cor}^5$};

\node [font=\large] at (-6.5,-17.25) {$c$};
\node [font=\large] at (-2.25,-17.25) {$bc$};
\draw [->, >=Stealth] (-6,-17.25) -- (-3,-17.25);

\node [font=\large] at (-0.5,-17.25) {$ac$};
\node [font=\large] at (6.5,-17.25) {$a'c'+b'c'$};
\draw [->, >=Stealth] (0.5,-17.25) -- (5.5,-17.25);

\node [font=\large] at (-0.5,-17.75) {$bc$};
\node [font=\large] at (6,-17.75) {$0$};
\node [font=\large] at (7.5,-17.75) {$a'c',b'c'$};
\draw [->, >=Stealth] (0.5,-17.75) -- (5.5,-17.75);
\draw [->, >=Stealth] (8.75,-17.75) -- (12.25,-17.75);
\draw [->, >=Stealth] (8.5,-17.25) -- (12.25,-17.25);
\node [font=\large] at (12.75,-17.75) {$ac$};
\node [font=\large] at (12.75,-17.25) {$0$};
\draw [->, >=Stealth] (13.5,0) .. controls (20,-4.5) and (-10,-3.5) .. (-4,-7.5);
\node [font=\large] at (7.5,-4.25) {$\cup (b)$};

      \end{circuitikz}
    }
    \caption{GLES $ac$ is negligible in $Q_{32}$}
    \label{Q32resQ16-1}
  \end{figure}

      \begin{figure}[htbp]      
    \centering
    \resizebox{\textwidth}{!}{%
      \begin{circuitikz}
                 \tikzstyle{every node}=[font=\large]
\node [font=\Large] at (-1.75,16.5) {$H^1(Q_{32}, \mathbb F_2)$};

\draw [->, >=Stealth] (0.5,16.5) -- (2.5,16.5);
\node [font=\Large] at (4.75,16.5) {$H^1(Q_{16}, \mathbb F_2)$};
\draw [->, >=Stealth] (7.25,16.5) -- (8.55,16.5);
\node [font=\Large] at (10.75,16.5) {$H^1(Q_{32}, \mathbb F_2)$};
\draw [->, >=Stealth] (13.5,16.5) .. controls (22.25,11) and (-10.75,15) .. (-4,9) ;
\node [font=\LARGE] at (3.25,16) {};
\node [font=\normalsize] at (1.75,16.9) {$\mathrm{res}^1$};
\node [font=\normalsize] at (7.75,16.9) {$\mathrm{cor}^1$};
\node [font=\large] at (0.25,15.5) {$b$};
\node [font=\large] at (5.25,15.5) {$a'+b'$};
\node [font=\LARGE] at (3.25,16) {};
\node [font=\LARGE] at (3.25,16) {};
\draw [->, >=Stealth] (0.75,15.5) -- (4.25,15.5);
\draw [->, >=Stealth] (6,15.5) -- (9.75,15.5);
\node [font=\large] at (10,15.5) {$0$};
\node [font=\large] at (0.25,15) {$a$};
\node [font=\large] at (4.75,15) {$0$};
\draw [->, >=Stealth] (0.75,15) -- (4.25,15);
\node [font=\large] at (5.75,15) {$a'$};
\node [font=\large] at (5.75,14.5) {$b'$};
\draw [->, >=Stealth] (6.25,14.5) -- (9.75,14.5);
\draw [->, >=Stealth] (6.25,15) -- (9.75,15);
\node [font=\large] at (10.25,15) {$b$};
\node [font=\large] at (10.25,14.5) {$b$};
\node [font=\large] at (0.25,6.5) {$a^2$};

\node [font=\large] at (7.25,13.5) {$\cup (a)$};

\node [font=\Large] at (-1.25,8.5) {$H^2(Q_{32}, \mathbb F_2)$};
\draw [->, >=Stealth] (1,8.5) -- (3,8.5);
\node [font=\Large] at (5.25,8.5) {$H^2(Q_{16}, \mathbb F_2)$};
\draw [->, >=Stealth] (7.75,8.5) -- (9,8.5);
\node [font=\Large] at (11.25,8.5) {$H^2(Q_{32}, \mathbb F_2)$};
\draw [->, >=Stealth] (14.25,8.5) .. controls (22.75,2.75) and (-10.25,6.75) .. (-3.75,0.5) ;
\node [font=\LARGE] at (3.75,8) {};
\node [font=\normalsize] at (2,8.75) {$\mathrm{res}^2$};
\node [font=\normalsize] at (8.25,8.75) {$\mathrm{cor}^2$};
\node [font=\large] at (0.25,7.5) {$b^2$};
\node [font=\large] at (5.75,7.5) {$a'^2+b'^2$};

\node [font=\LARGE] at (3.75,8) {};
\node [font=\LARGE] at (3.7-5,8) {};
\draw [->, >=Stealth] (1.25,7.5) -- (4.75,7.5);

\node [font=\large] at (12.25,7.5) {$b^2$};
\node [font=\large] at (7.5,5.25) {$\cup (a)$};
\node [font=\large] at (7.75,7.5) {$a',b'$};
\draw [->, >=Stealth] (8.25,7.5) -- (11.75,7.5);

\node [font=\large] at (-7,7.5) {$a$};

\node [font=\large] at (-2.1,7.5) {$a^2$};

\draw [->, >=Stealth] (-6.5,7.5) -- (-3,7.5);
\node [font=\large] at (-7,7) {$b$};
\node [font=\large] at (-2.1,7) {$0$};
\draw [->, >=Stealth] (-6.5,7) -- (-3,7);

\draw [->, >=Stealth] (1.25,6.5) -- (4.75,6.5);
\node [font=\large] at (5.25,6.5) {$0$};

\node [font=\Large] at (-1,0.25) {$H^3(Q_{32}, \mathbb F_2)$};
\draw [->, >=Stealth] (1.5,0.25) -- (3.25,0.25);
\node [font=\Large] at (5.5,0.25) {$H^3(Q_{16}, \mathbb F_2)$};
\draw [->, >=Stealth] (8,0.25) -- (9.25,0.25);
\node [font=\Large] at (11.5,0.25) {$H^3(Q_{32}, \mathbb F_2)$};
\node [font=\Large] at (4,-0.25) {};
\node [font=\normalsize] at (2.5,0.5) {$\mathrm{res}^3$};
\node [font=\normalsize] at (8.5,0.5) {$\mathrm{cor}^3$};

\node [font=\large] at (5.5,-0.75) {$0$};
\node [font=\large] at (5.5,-1.75) {$0$};

\node [font=\LARGE] at (4,-0.25) {};
\node [font=\LARGE] at (4,-0.25) {};
\draw [->, >=Stealth] (-1,-0.75) -- (5,-0.75);

\node [font=\large] at (-1.75,-0.75) {$a^3=b^3$};
\draw [->, >=Stealth] (-6.25,-0.75) -- (-2.75,-0.75);

\draw [->, >=Stealth] (-1,-1.75) -- (5,-1.75);

\node [font=\large] at (-6.5,-0.75) {$a^2$};

\node [font=\large] at (-6.5,-1.75) {$b^2$};

\node [font=\large] at (6.75,-1.75) {$a'^3=b'^3$};
\node [font=\large] at (10.75,-1.75) {$b^3$};
\draw [->, >=Stealth] (8,-1.75) -- (10,-1.75);
\draw [->, >=Stealth] (-6.25,-1.75) -- (-2.75,-1.75);
\node [font=\large] at (-2,-1.75) {$0$};
\node [font=\Large] at (-1, -8) {$H^4(Q_{32}, \mathbb F_2)$};
\draw [->, >=Stealth] (1.5,-8) -- (3.5,-8);
\node [font=\Large] at (6,-8) {$H^4(Q_{16}, \mathbb F_2)$};
\draw [->, >=Stealth] (8.5,-8) -- (10,-8);
\node [font=\Large] at (12.5,-8) {$H^4(Q_{32}, \mathbb F_2)$};
\node [font=\normalsize] at (2.5,-7.5) {$\mathrm{res}^4$};
\node [font=\normalsize] at (9,-7.5) {$\mathrm{cor}^4$};

\node [font=\large] at (-0.5,-9.25) {$c$};
\node [font=\large] at (6,-9.25) {$c'$};
\draw [->, >=Stealth] (-0.25,-9.25) -- (5.5,-9.25);
\draw [->, >=Stealth] (6.25,-9.25) -- (11,-9.25);
\node [font=\large] at (11.5,-9.25) {$0$};

\draw [->, >=Stealth] (14.5,-8) .. controls (20,-12.5) and (-10,-13) .. (-4,-16);
\node [font=\large] at (7.5,-11.5) {$\cup (a)$};

\node [font=\Large] at (-1, -16.5) {$H^5(Q_{32}, \mathbb F_2)$};
\draw [->, >=Stealth] (1.5,-16.5) -- (3.5,-16.5);
\node [font=\Large] at (6,-16.5) {$H^5(Q_{16}, \mathbb F_2)$};
\draw [->, >=Stealth] (8.5,-16.5) -- (10,-16.5);
\node [font=\Large] at (12.5,-16.5) {$H^5(Q_{32}, \mathbb F_2)$};
\node [font=\normalsize] at (2.5,-16) {$\mathrm{res}^5$};
\node [font=\normalsize] at (9,-16) {$\mathrm{cor}^5$};

\node [font=\large] at (-6.5,-17.25) {$c$};
\node [font=\large] at (-2.25,-17.25) {$ac$};
\draw [->, >=Stealth] (-6,-17.25) -- (-3,-17.25);

\node [font=\large] at (-0.5,-17.25) {$bc$};
\node [font=\large] at (6.5,-17.25) {$a'c'+b'c'$};
\draw [->, >=Stealth] (0.5,-17.25) -- (5.5,-17.25);

\node [font=\large] at (-0.5,-17.75) {$ac$};
\node [font=\large] at (6,-17.75) {$0$};
\node [font=\large] at (7.5,-17.75) {$a'c',b'c'$};
\draw [->, >=Stealth] (0.5,-17.75) -- (5.5,-17.75);
\draw [->, >=Stealth] (8.75,-17.75) -- (12.25,-17.75);
\draw [->, >=Stealth] (8.5,-17.25) -- (12.25,-17.25);
\node [font=\large] at (12.75,-17.75) {$bc$};
\node [font=\large] at (12.75,-17.25) {$0$};
\draw [->, >=Stealth] (13.5,0) .. controls (20,-4.5) and (-10,-3.5) .. (-4,-7.5);
\node [font=\large] at (7.5,-3.25) {$\cup (a)$};

      \end{circuitikz}
    }
    \caption{GLES $bc$ is negligible in $Q_{32}$}
    \label{Q32resQ16-2}
  \end{figure}

\end{document}